\documentclass[pamq,nonatbib,seceqn]{ipart}
\pdfoutput=1
\startlocaldefs
\usepackage[ascii]{inputenc}
\usepackage[bookmarksnumbered]{hyperref}
\usepackage{amsmath,amsthm,amssymb,mathtools,braket,enumitem}
\usepackage{tikz}
\usetikzlibrary{positioning}
\usepackage[T1]{fontenc}
\usepackage[USenglish]{babel}
\usepackage{amsrefs}
\usepackage[capitalize]{cleveref}
\newtheorem{theorem}{Theorem}[section]
\newtheorem{proposition}[theorem]{Proposition}
\newtheorem{lemma}[theorem]{Lemma}
\newtheorem{definition}[theorem]{Definition}
\newtheorem{corollary}[theorem]{Corollary}
\newtheorem{remark}[theorem]{Remark}
\newtheorem{example}[theorem]{Example}
\newtheorem{claim}[theorem]{Claim}\crefname{claim}{Claim}{Claims}
\DeclareMathOperator{\Belkale}{Belkale}
\DeclareMathOperator{\coker}{coker}
\DeclareMathOperator{\edim}{edim}
\DeclareMathOperator{\eul}{eul}
\DeclareMathOperator{\Ext}{Ext}
\DeclareMathOperator{\ext}{ext}
\DeclareMathOperator{\GL}{GL}
\DeclareMathOperator{\Gr}{Gr}
\DeclareMathOperator{\Hom}{Hom}
\DeclareMathOperator{\Horn}{Horn}
\DeclareMathOperator{\Sym}{Sym}
\newcommand{\BS}{{\mathbb S}}
\newcommand{\C}{{\mathbb C}}

\newcommand{\N}{{\mathbb N}}

\newcommand{\R}{{\mathbb R}}

\newcommand{\ZZ}{{\mathbb Z}}
\newcommand{\Z}{{\mathbb Z}}
\newcommand{\boldk}{\boldsymbol k}
\newcommand{\boldn}{\boldsymbol n}
\newcommand{\boldalpha}{\boldsymbol\alpha}
\newcommand{\boldbeta}{\boldsymbol\beta}
\newcommand{\bolddelta}{\boldsymbol\delta}
\newcommand{\boldlambda}{\boldsymbol\lambda}
\newcommand{\boldomega}{\boldsymbol\omega}
\newcommand{\boldtilden}{\boldsymbol{\tilde n}}
\newcommand{\boldtildealpha}{\boldsymbol{\tilde\alpha}}
\newcommand{\boldtildebeta}{\boldsymbol{\tilde\beta}}
\newcommand{\boldtildeomega}{\boldsymbol{\tilde\omega}}
\newcommand{\boldOmega}{\boldsymbol\Omega}
\newcommand{\boldhatOmega}{\boldsymbol{\hat\Omega}}
\newcommand{\gl}{{\mathfrak{gl}}}
\newcommand{\g}{{\mathfrak{g}}}
\newcommand{\z}{{\mathfrak{z}}}
\renewcommand{\u}{{\mathfrak{u}}}
\newcommand{\CF}{{\mathcal F}}
\newcommand{\CG}{{\mathcal G}}
\newcommand{\CH}{{\mathcal H}}
\newcommand{\CJ}{{\mathcal J}}
\newcommand{\CK}{{\mathcal K}}
\newcommand{\CL}{{\mathcal L}}
\newcommand{\CM}{{\mathcal M}}
\newcommand{\CN}{{\mathcal N}}

\newcommand{\CS}{{\mathcal S}}
\newcommand{\CT}{{\mathcal T}}
\newcommand{\CU}{{\mathcal U}}
\newcommand{\CV}{{\mathcal V}}
\newcommand{\CW}{{\mathcal W}}
\renewcommand{\b}{{\mathfrak{b}}}
\allowdisplaybreaks[4]
\endlocaldefs
\pubyear{2023}
\volume{0}
\issue{0}
\firstpage{1}
\lastpage{1}
\arxiv{1901.07194}
\begin{document}
\begin{frontmatter}
\title[Horn conditions for quivers and the moment map]{Horn conditions for quiver subrepresentations and~the~moment~map\protect\thanksref{T1}}
\thankstext{T1}{We are delighted to include this article as a tribute to the always inspiring work of Victor Guillemin.}
\begin{aug}
    \author{\fnms{Velleda} \snm{Baldoni}\ead[label=e1]{baldoni@mat.uniroma2.it}},
    \address{Dipartimento di Matematica, Universit\`a degli studi di Roma ``Tor Vergata'' \\
             Via della ricerca scientifica 1, 00133 Roma, Italy\\
             \printead{e1}}
    \author{\fnms{Mich\`ele} \snm{Vergne}\ead[label=e2]{michele.vergne@imj-prg.fr}},
    \address{Institut de Math\'ematiques de Jussieu-Paris Rive Gauche, Universit\'e Paris Cit\'e\!\!\!\!\!\! \\
             B\^atiment Sophie Germain, 8 place Aur\'elie Nemours, Boite Courrier 7012, 75205 Paris Cedex 13, France \\
             \printead{e2}}
    \and
    \author{\fnms{Michael} \snm{Walter}
            \ead[label=e3]{michael.walter@rub.de}}
    \address{Faculty of Computer Science, Ruhr University Bochum\\
             Universit\"atsstr.\ 150, 44801 Bochum, Germany\\
             \textit{and}\\
             Korteweg-de Vries Institute for Mathematics, Institute for Theoretical Physics, Institute for Logic, Language and Computation, and QuSoft \\
             University of Amsterdam, 1098 XG Amsterdam, Netherlands \\
             \printead{e3}}
\end{aug}
\received{\sday{18} \smonth{7} \syear{2021}}
\begin{abstract}
We give inductive conditions that characterize the Schubert positions of subrepresentations of a general quiver representation.
Our results generalize Belkale's criterion for the intersection of Schubert varieties in Grassmannians and refine Schofield's characterization of the dimension vectors of general subrepresentations.
This implies Horn type inequalities for the moment cone associated to the linear representation of the group $G=\prod_x \GL(n_x)$ associated to a quiver and a dimension vector $\boldn=(n_x)$.
\end{abstract}
\begin{keyword}[class=AMS]
\kwd[Primary ]{53D20} 
\kwd{14N15} 
\kwd[; secondary ]{15A42} 
\kwd{16G20} 
\kwd{22E47} 
\end{keyword}
\end{frontmatter}

\section{Introduction}\label{sec:intro}
Let $Q=(Q_0,Q_1)$ be a quiver, where~$Q_0$ is the finite set of vertices and~$Q_1$ the finite set of arrows.
We use the notation~$a\colon x\to y$ for an arrow~$a\in Q_1$ from~$x\in Q_0$ to~$y\in Q_0$.
We allow $Q$ to have cycles and multiple arrows between two vertices.
A \emph{dimension vector} for~$Q$ is a vector~$\boldn = (n_x)_{x\in Q_0}$ of nonnegative integers.

To every family of vector spaces~$\CV=(V_x)_{x\in Q_0}$, we associate the dimension vector $\dim\CV$ with components~$(\dim \CV)_x = \dim V_x$.
The space of \emph{representations} of the quiver~$Q$ on~$\CV$ is given by
\begin{align}\label{eq:quiver repr}
  \CH_Q(\CV) \coloneqq \bigoplus_{a:x\to y\in Q_1} \Hom(V_x, V_y),
\end{align}
whose elements are families~$v=(v_a)_{a\in Q_1}$ of linear maps~$v_a: V_x\to V_y$, one for each arrow $a\colon x\to y$. 
The Lie group~$\GL_Q(\CV) = \prod_{x\in Q_0} \GL(V_x)$ and its Lie algebra~$\gl_Q(\CV) = \bigoplus_{x\in Q_0} \gl(V_x)$ act naturally on~$\CV$ and on~$\CH_Q(\CV)$.
For~$g = (g_x)_{x\in Q_0} \in \GL_Q(\CV)$ and~$X = (X_x)_{x\in Q_0} \in \gl_Q(\CV)$, their actions on~$u = (u_x)_{x\in Q_0} \in \CV$ are given by~$gu \coloneqq (g_x u_x)_{x\in Q_0}$ and~$Xu \coloneqq (X_x u_x)_{x\in Q_0}$, respectively, while their actions on $v = (v_a)_{a\in Q_1} \in \CH_Q(\CV)$ are denoted by~$gvg^{-1} \coloneqq (g_y v_a g_x^{-1})_{a\colon x\to y \in Q_1}$ and~$Xv-vX \coloneqq (X_y v_a - v_a X_x)_{a\colon x\to y \in Q_1}$, respectively.

We write $\CS\subseteq\CV$ if $\CS=(S_x)_{x\in Q_0}$ is a family of subspaces~$S_x \subseteq V_x$; its dimension vector is called a \emph{subdimension vector} for $\CV$, i.e., satisfies $\dim S_x \leq \dim V_x$.
The family~$\CS$ is called a \emph{subrepresentation} of~$v\in\CH_Q(\CV)$ if $v_a S_x\subseteq S_y$ for every arrow $a\colon x\to y$ in~$Q_1$;
we abbreviate this condition by $v \CS \subseteq \CS$.

Schofield~\cite{MR1162487} characterized (inductively) the subdimension vectors~$\boldalpha$ such that any~$v\in\CH_Q(\CV)$ has a subrepresentation~$\CS$ with $\dim\CS = \boldalpha$.
We call such a dimension vector a \emph{Schofield subdimension vector} for~$\CV$ and denote this by~$\boldalpha\leq_Q \boldn$, where~$\dim\CV = \boldn$.
We also write $\boldalpha <_Q \boldn$ if in addition at least one of the inequalities $\alpha_x \leq n_x$ is strict.
As the notation suggests, these relations are transitive.

Consider
\begin{align*}
  \Gr_Q(\boldalpha,\CV) \coloneqq \prod_{x\in Q_0} \Gr(\alpha_x, V_x),
\end{align*}
where $\Gr(\alpha_x,V_x)$ denotes the Grassmannian of subspaces of~$V_x$ of dimension~$\alpha_x$.
The dimension of $\Gr_Q(\boldalpha,\CV)$ is given by $\sum_{x\in Q_0} \alpha_x \beta_x$, where~$\beta_x = \dim V_x - \alpha_x$.

Given a representation~$v\in\CH_Q(\CV)$ and a dimension vector~$\boldalpha$, we define the corresponding \emph{quiver Grassmannian} by
\begin{align*}
  \Gr_Q(\boldalpha,\CV)_v \coloneqq \{ \CS \in \Gr_Q(\boldalpha) : v\CS \subseteq \CS \}.
\end{align*}
In this language, a Schofield subdimension vector is a subdimension vector~$\boldalpha$ such that $\Gr_Q(\boldalpha,\CV)_v \neq\emptyset$ for every representation~$v\in\CH_Q(\CV)$.
In this case, the dimension of each irreducible component of the quiver Grassmannian~$\Gr_Q(\boldalpha,\CV)_v$ is, for generic $v\in\CH_Q(\CV)$, given by
\begin{align}\label{eq:euler form}
  \braket{\boldalpha,\boldbeta}
\coloneqq \sum_{x\in Q_0} \alpha_x \beta_x - \!\!\!\!\sum_{a:x\to y\in Q_1}\!\!\!\! \alpha_x\beta_y.
\end{align}
Thus, the codimension of~$\Gr_Q(\boldalpha,\CV)_v$ in~$\Gr_Q(\boldalpha,\CV)$ is $\sum_{a: x\to y\in Q_1}\alpha_x\beta_y$.

\subsection{Schubert varieties and \texorpdfstring{$Q$}{Q}-intersection}
It is natural to study the possible Schubert positions of quiver subrepresentations.
For this purpose, we introduce the notion of filtered dimension vector (partly inspired by the augmented quivers of Derksen-Weyman).

Fix a family $\CF=(F_x)_{x\in Q_0}$ of (complete) filtrations, where each~$F_x$ is a (complete) filtration on~$V_x$.
We call $(\CV,\CF)$ a \emph{filtered dimension vector} (see \cref{sec:defs}).
Let~$B_Q(\CV,\CF)=(B_x)_{x\in Q_0}$ denote the corresponding Borel subgroup of $\GL_Q(\CV)$, i.e., each~$B_x \subseteq \GL(V_x)$ is the Borel subgroup preserving the filtration~$F_x$.
Finally, let~$\boldOmega=(\Omega_x)_{x\in Q_0}$ be a \emph{Schubert variety} in~$\Gr_Q(\boldalpha,\CV)$, i.e.,~$\boldOmega$ is the closure of a $B_Q(\CV,\CF)$-orbit in~$\Gr_Q(\boldalpha,\CV)$.
Then we say that $\boldOmega$ is \emph{$Q$-intersecting} (in~$\CV$) if the intersection
\begin{align}\label{eq:intersection variety}
  \boldOmega_v \coloneqq \boldOmega \cap \Gr_Q(\boldalpha,\CV)_v
\end{align}
is nonempty for every $v\in\CH_Q(\CV)$.
In other words, $\boldOmega$ is $Q$-intersecting if every quiver representation on~$\CV$ has a subrepresentation in the Schubert variety~$\boldOmega$.
When $\boldOmega = \Gr_Q(\boldalpha,\CV)$ is the largest Schubert variety, then $\boldOmega$ is $Q$-intersecting if and only if $\boldalpha$ is a Schofield subdimension vector.
Thus, $Q$-intersection is a more refined notion.
The main result of this article is an inductive family of necessary and sufficient conditions for $\boldOmega$ to be $Q$-intersecting (\cref{thm:main} below).

An important example is the \emph{Horn quiver}~$H_s$, which has $s+1$~vertices and $s$~arrows:
\begin{equation}\label{eq:horn quiver}
\begin{aligned}
\begin{tikzpicture}\small
  \node (O) at (0,1) {$s+1$};
  \node (A) at (-1.5,0) {$1$};
  \node at (0,0) {\dots};
  \node (Z) at (1.5,0) {$s$};
  \draw[->] (A) edge (O);
  \draw[->] (Z) edge (O);
\end{tikzpicture}
\end{aligned}
\end{equation}
Let $0\leq r\leq n$, $V_x=\C^n$, and $\alpha_x=r$ for $x=1,\dots,s+1$.
Then, a Schubert variety $\boldOmega\subseteq \Gr_Q(\boldalpha,\CV)$ is an $(s+1)$-tuple of Schubert varieties~$\Omega_1,\dots,\Omega_s,\Omega_{s+1}$ in $\Gr(r,n)$.
The condition that $\boldOmega$ is $Q$-intersecting is equivalent to the condition that the Schubert homology classes $[\Omega_x]_{x=1}^{s+1}$ are intersecting (\cref{ex:horn quiver}).
Horn~\cite{MR0140521} suggested necessary and sufficient conditions for Schubert varieties to intersect.
The validity of Horn's criterion was established by Knutson-Tao~\cite{MR1671451} using a combinatorial approach that established the saturation conjecture for the Littlewood-Richardson coefficients.
Derksen-Weyman~\cite{MR1758750} gave an alternative proof using the theory of quiver representations, which was further simplified by Crawley-Boevey-Geiss~\cite{Cra-Bo-Ge} (see \cref{subsec:schofield} below).
Finally, Belkale~\cite{MR2177198} gave a geometric proof of a strengthened version of the Horn criterion and the saturation conjecture.

As in~\cite{MR2177198}, our inductive criterion for $\boldOmega$ to be $Q$-intersecting is based on a numerical quantity:
the expected dimension of the intersection variety~$\boldOmega_v$ defined in~\eqref{eq:intersection variety}.
Since the codimension of~$\Gr_Q(\boldalpha,\CV)_v$ in~$\Gr_Q(\boldalpha,\CV)$ is generically equal to $\sum_{a: x\to y\in Q_1}\alpha_x\beta_y$, the `expected dimension' of the intersection is given by
\begin{align*}
  \edim_{Q,\CF}(\boldOmega,\CV) \coloneqq \dim\boldOmega - \!\!\!\!\sum_{a\colon x\to y\in Q_1}\!\!\!\! \alpha_x \beta_y.
\end{align*}
It is easy to prove that if $\boldOmega$ is $Q$-intersecting then, for generic~$v$, the dimension of the intersection variety $\boldOmega_v$ is indeed equal to the expected dimension.
Thus, $\edim_{Q,\CF}(\boldOmega,\CV)\geq0$ is a necessary condition for $\boldOmega$ to be $Q$-intersecting.
However, this necessary condition is not sufficient (a simple example is given below in \cref{subsec:counter}).

Before giving a complete set of conditions we introduce some convenient notation.
Given a family of subspaces~$\CS\subseteq\CV$, we denote by~$\boldOmega(\CS,\CF)$ the Schubert variety determined by~$\CS$, i.e., the closure of the $B_Q(\CV,\CF)$-orbit of~$\CS$, and we let $\edim_{Q,\CF}(\CS,\CV)=\edim_{Q,\CF}(\boldOmega(\CS,\CF),\CV).$
We say~$\CS$ is~\emph{$Q$-intersecting in~$\CV$} if~$\boldOmega(\CS,\CF)$ is $Q$-intersecting in~$\CV$.
That is, for generic~$v\in \CH_Q(\CV)$,~$\CS$~is a subrepresentation of some point in the $B_Q(\CV,\CF)$-orbit of~$v$.
We denote this condition by $\CS \subseteq_Q \CV$, and write $\CS \subset_Q \CV$ if at least one $S_x$ is a proper subspace of~$V_x$.

As explained above, a necessary condition for~$\CS$ to be $Q$-intersecting in~$\CV$ is that~$\edim_{Q,\CF}(\CS,\CV)\geq0$.
It is also easy to see that the relation $\subseteq_Q$ is transitive (\cref{lem:transitive}):
if $\CT\subseteq_Q\CS$ and $\CS\subseteq_Q \CV$, then $\CT\subseteq_Q \CV$.
Our main result is that these two natural conditions are not only necessary but also sufficient:

\begin{theorem}\label{thm:main}
Let $\CV$ be a family of vector spaces, $\CF$ a family of filtrations, and $\CS$ a family of subspaces of~$\CV$.
Then, $\CS \subseteq_Q \CV$ if and only if
\begin{enumerate}[label=\emph{(\Alph*)},ref={(\Alph*)}]
\item\label{it:main A} $\edim_{Q,\CF}(\CS,\CV) \geq 0$,
\item\label{it:main B} $\CT \subset_Q \CV$ for every $\CT \subset_Q \CS$.
\end{enumerate}
\end{theorem}

In fact, we obtain slightly stronger results than \cref{thm:main}.
In conditions~\ref{it:main B}, we merely need to consider those~$\CT \subset_Q \CS$ such that the generic intersection variety is a point (\cref{thm:main refined}).

\Cref{thm:main} generalizes Belkale's criterion for the intersection of Schubert varieties in Grassmannians, and we believe that working in this general context elucidates the arguments.
The first ingredient of our proof is a generalization of Schofield's numerical computation~\cite{MR1162487} of the dimension of certain $\Ext$-groups to the filtered setting (\cref{thm:ext via eul}).
Here we follow closely (but do not rely on) Schofield's argument.
We note that an alternative proof of \cref{thm:ext via eul} was recently given by Bertozzi-Reineke~\cite{BR} using augmented quivers (see \cref{subsec:schofield}); however, their results do not imply \cref{thm:main}.
To obtain the simple inductive characterization in our main result, \cref{thm:main}, we use an argument on slopes inspired by Harder-Narasimhan filtration, adapting an argument of Belkale (\cref{sec:6}).

\subsection{Example}\label{subsec:counter}
To illustrate \cref{thm:main}, consider the following quiver:
\begin{equation}\label{eq:square quiver}
\raisebox{-.8cm}{\begin{tikzpicture}\small
  \node (x1) at (0,1.5) {$1$};
  \node (x2) at (1.5,1.5) {$2$};
  \node (x3) at (0,0) {$3$};
  \node (x4) at (1.5,0) {$4$};
  \draw[->] (x1) edge (x2);
  \draw[->] (x1) edge (x3);
  \draw[->] (x2) edge (x4);
  \draw[->] (x3) edge (x4);
\end{tikzpicture}}
\end{equation}
Let~$(\CV,\CF)$ be the filtered dimension vector with~$V_1=V_4=\C^2$ and $V_2=V_3=\C^3$, and where $F_x$ is the standard filtration for every vertex~$x$.
Then there are 172~$Q$-intersecting Schubert varieties, corresponding to 46~Schofield subdimension vectors.

For example, $\CS=(\C e_1, \C e_2 \oplus \C e_3, \C e_2 \oplus \C e_3, \C^2)$ is $Q$-intersecting, while $\smash{\hat\CS}=(\C e_1, \C e_2 \oplus \C e_3, \C e_1 \oplus \C e_2, \C^2)$ is not.
This is easy to see directly, since the associated Schubert varieties are
\begin{align*}
  \boldOmega&=(\{\C e_1\}, \Gr(2,3), \Gr(2,3), \{\C^2\}), \\
  \boldhatOmega&=(\{\C e_1\}, \Gr(2,3), \{\C e_1 \oplus \C e_2\}, \{\C^2\}),
\end{align*}
respectively, and for a generic representation~$v\in\CH_Q(\CV)$ the component~$v_{1\to3}$ does not map~$e_1$ into $\C e_1 \oplus \C e_2$.
Now, note that
\begin{align*}
  \edim_{Q,\CF}(\CS,\CV)&=\dim \boldOmega-(1+1)=(0+2+2+0)-(1+1)>0, \\
  \edim_{Q,\CF}(\smash{\hat\CS},\CV)&=\dim \boldhatOmega-(1+1)=(0+2+0+0)-(1+1)=0,
\end{align*}
so condition~\ref{it:main A} of \cref{thm:main} is satisfied for both~$\CS$ and $\smash{\hat\CS}$.
Thus, condition~\ref{it:main B} must be violated for~$\smash{\hat\CS}$, so there exists a family~$\CT$ of proper subspaces which is $Q$-intersecting in $\smash{\hat\CS}$, but not in~$\CV$.
Indeed, the family $\CT = (\C e_1, \C e_3, \C e_2, \C^2)$ has this property.
We discuss a more involved example involving Collins' `sun quiver'~\cite{Collins} in \cref{sec:sun quiver}.

\subsection{An inductive numerical criterion and Horn-type inequalities}\label{subsec:horn}
Inductively, \cref{thm:main} translates into the following criterion:

\begin{theorem}\label{thm:numercs}
$\CS \subseteq_Q \CV$ if and only if $\edim_{Q,\CF}(\CT,\CV) \geq 0$ for all $\CT \subseteq_Q \CS$.
\end{theorem}

Note that \cref{thm:numercs} amounts to a finite criterion since the right-hand side depends only on the Schubert variety determined by $\CT$, of which there are only finitely many.
This can be made particularly concrete by parameterizing the Schubert cells, which also makes the connection to Belkale's Horn-type inequalities directly apparent.

Let $\boldn=(n_x)_{x\in Q_0}$ be a dimension vector, and let $\CV$ be the family of standard complex vector spaces $V_x = \C^{n_x}$, equipped with the standard filtrations.
Any family $\CK \subseteq [\boldn]$, by which we mean that~$\CK=(K_x)_{x\in Q_0}$ consists of subsets $K_x \subseteq \{1,\dots,n_x\}$, determines a family~$\CS=(S_x)_{x\in Q_0}$ of subspaces $S_x = \oplus_{i \in K_x} e_i$, where $e_i$ denotes the standard basis of~$V_x$, with dimension vector $\boldsymbol k = (k_x)_{x\in Q_0} = (\lvert K_x\rvert)_{x \in Q_0}$, and hence a Schubert variety~$\boldsymbol\Omega$.
Any Schubert variety can be obtained in this way.
It is easy to see that if $K_x(1) < \dots < K_x(k_x)$ are the elements of~$K_x$ then the dimension of the Schubert variety determined by~$\CK$ is
\begin{align*}
  \dim\boldsymbol\Omega = \sum_{x\in Q_0} \sum_{j=1}^{k_x} \left( K_x(j) - j \right).
\end{align*}
Let us write $\CK \subseteq_Q [\boldn]$ to denote that $\CS \subseteq_Q \CV$.
Now, any family $\CL \subseteq [\boldsymbol k]$ rise to a family $\CT = (T_x)_{x\in Q_0}$ of subspaces $T_x = \oplus_{j=1}^{l_x} e_{K_x(L_x(j))} \subseteq S_x$, where $L_x(1) < \dots < L_x(l_x)$ are the elements of $L_x$ and $l_x = \lvert L_x\rvert$.
We may calculate that
\begin{align*}
  \edim_{Q,\CF}(\CT,\CV)
= \sum_{x\in Q_0} \sum_{j=1}^{l_x} \left( K_x(L_x(j)) - j \right)
\;-\!\!\!\! \sum_{a : x\to y \in Q_1} l_x (n_y - l_y).
\end{align*}
Accordingly, \cref{thm:numercs} translates into the following inductive numerical criterion:
$\CK \subseteq_Q [\boldsymbol n]$ if and only if
\begin{align}\label{eq:refined S}
  \sum_{x\in Q_0} \sum_{j=1}^{l_x} \left( K_x(L_x(j)) - j \right) \;\;\geq \!\!\!\!\sum_{a : x\to y \in Q_1} l_x (n_y - l_y)
\end{align}
for all $\CL \subseteq_Q [\boldsymbol k]$.
In the case of the Horn quiver, we recognize Belkale's inequalities.
The criterion in \cref{eq:refined S} is easy to test numerically.
We note that one may further restrict the families~$\CL$ that need to be considered (\cref{rem:opt test 2}).

\subsection{A natural Schofield criterion and augmented quivers}\label{subsec:schofield}
As a particular consequence of \cref{thm:main} we also obtain the following inductive characterization of Schofield subdimension vectors:

\begin{theorem}\label{thm:mainscalar}
Let $\boldalpha \leq \boldn$ be dimension vectors.
Then, $\boldalpha \leq_Q \! \boldn$ if and only~if
\begin{enumerate}[label=\emph{(\Alph*)},ref={(\Alph*)}]
\item\label{it:main As} $\braket{\boldalpha,{\boldn}-\boldalpha}\geq 0$,
\item\label{it:main Bs} $\boldbeta <_Q \boldn$ for every $\boldbeta <_Q \boldalpha$.
\end{enumerate}
\end{theorem}

Just like for \cref{thm:main}, we obtain in fact a slightly stronger characterization by restricting part~\ref{it:main Bs} to those $\boldbeta$'s for which the generic intersection variety is a point (\cref{thm:schofield refined}).
\Cref{thm:mainscalar} is readily translated into the following inductive numerical criterion:

\begin{theorem}\label{theo:numericalscalar}
$\boldalpha \leq_Q \boldn$
if and only if
$\braket{\boldbeta,\boldn-\boldbeta} \geq 0$ for all $\boldbeta \leq_Q \boldalpha$.
\end{theorem}

We note that \cref{theo:numericalscalar} does not follow right away from the Schofield criterion~\cite{MR1162487}, despite the latter looking very similar:
$\boldalpha \leq_Q \boldn$ if and only if $\braket{\boldbeta,\boldn-\boldalpha} \geq 0$ for all $\boldbeta \leq_Q \boldalpha$.
Note that the condition on~$\beta$ coincides with ours if and only if $\braket{\boldbeta,\boldalpha-\boldbeta}=0$.
Indeed, it follows from the strengthening of \cref{thm:mainscalar} discussed above that it suffices to restrict to such~$\boldbeta$ in order to characterize Schofield subdimension vectors.
To obtain the natural inductive characterization given in \cref{thm:mainscalar} (and its strengthening) or, equivalently, the numerical criterion of \cref{theo:numericalscalar}, we found it necessary to use a slope argument (see \cref{sec:6} for the more general filtered setting).

Derksen-Weyman~\cite{MR1758750} deduced the Horn inequalities for tensor products using an `augmented' quiver~$\tilde Q$ associated to~$Q$.
To see the relation, given a filtered dimension vector $(\CV,\CF)$, we define by~$\tilde n_{x,i} = \dim F_x(i)$ an ordinary dimension vector~$\boldtilden$ on an augmented quiver~$\tilde Q$ with vertices~$(x,i)$ for~$x\in Q_0$ and $i=1,\dots,\ell_x$, where $\ell_x$ denotes the length of the filtration~$F_x$.
Given a family of subspaces $\CS \subseteq \CV$, consider the subdimension vector $\boldtildealpha$ with $\alpha_{x,i} = \dim S_x \cap F_{x,i}$.
Then, $\CS \subseteq_Q \CV$ if and only if $\boldtildealpha \leq_{\tilde Q} \boldtilden$, so one could use Schofield's criterion or our inductive conditions for Schofield subdimension vectors to characterize $Q$-intersection.
However, the resulting criterion for $Q$-intersection is arguably less natural than our \cref{thm:main}, and it is also weaker, since in general there are in general many more subdimension vectors~$\boldtildebeta <_{\tilde Q} \boldtildealpha$ than $Q$-intersecting subfamilies $\CT \subset_Q \CS$.
We comment on the relation between the two sets in \cref{subsec:augmented}.

\subsection{Applications to representation theory and the moment map}\label{subsec:intro rep theo}
Another motivation to study $Q$-intersection comes from representation theory and symplectic geometry.
Indeed, if $K$ is a compact connected Lie group, the celebrated $[Q,R]=0$ or ``quantization commutes with reduction'' conjecture of Guillemin-Sternberg relates the quantization of a $K$-Hamiltonian manifold~$M$ to the image of the associated moment map.
In their original article, Guillemin-Sternberg established this conjecture when $M$ is a (smooth) compact K\"ahler manifold~\cite{GS1982qr}.
In the case of a projective variety~$M=\mathbb P(\mathcal C)$ associated to an algebraic cone~$\mathcal C$ invariant under a linear representation of a complex reductive group~$G$, Mumford's construction of the geometric quotient directly describes the action of~$G$ on polynomial functions on~$\mathcal C$ in terms of the moment map on~$\mathcal C$ associated to a compact form~$K$ of~$G$~\cite{NessMumford84}*{Appendix}.
In both cases, it follows that the image under the $K$-moment map of~$M$ (resp.~of~$\mathcal C$) modulo the coadjoint action of~$K$  is a rational convex polytope (resp.~a rational convex polyhedral cone), see also \cite{GS1982convex}*{Appendix}.
It is in general a difficult problem to describe these moment polytopes or cones explicitly and effectively.

Here we plainly consider the action of $G=\GL_Q(\CV)$ on the complex vector space~$\mathcal C = \CH_Q(\CV)$.
Let~$C_Q(\CV)$ denote the polyhedral cone spanned by the highest weights of irreducible representations of~$\GL_Q(\CV)$ that occur with nonzero multiplicity in~$\Sym^*(\CH_Q(\CV))$, the space of polynomial functions on~$\CH_Q(\CV)$.
Our aim is to describe this cone by inequalities associated to quiver subrepresentations.
It follows from the general theory described above that $C_Q(\CV)$ is the moment cone associated with a natural moment map and we will come back to this point momentarily.

The subcone $\Sigma_Q(\CV) \subseteq C_Q(\CV)$ generated by the weights of semi-invariants polynomials is of particular interest for invariant theory and moduli spaces of quiver representations (see King~\cite{MR1315461}, or Crawley-Boevey~\cite{MR1834739} for the double quiver case).
Derksen-Weyman~\cite{MR1758750} and Schofield-van den Bergh~\cite{MR1908144} showed that $\boldomega=(\omega_x)_{x\in Q_0}$ is a weight of a nonzero semi-invariant polynomial on~$\CH_Q(\CV)$ (that is, a polynomial that transforms by the character~$g=(g_x) \mapsto \prod_x \det(g_x)^{\omega_x}$ of~$\GL_Q(\CV)$) if and only if
\begin{align*}
  \sum_{x\in Q_0} n_x \omega_x = 0
\end{align*}
and, for all $\boldalpha <_Q \boldn$,
\begin{align*}
  \sum_{x\in Q_0} \alpha_x \omega_x\leq 0,
\end{align*}
where $n_x = \dim V_x$ for $x\in Q_0$.
Thus, the cone $\Sigma_Q(\CV)$ is determined by inequalities associated to Schofield subdimension vectors.

Similarly, the cone $C_Q(\CV)$ is determined by the $Q$-intersection of Schubert varieties and hence by our \cref{thm:main}.
Choose a Hermitian structure on~$V_x$, and let $U(V_x)$ be the maximally compact subgroup of $GL(V_x)$ consisting of unitary operators, with Lie algebra~$\u_x$.
We may identify $\sqrt{-1}\u_x$ with the space of Hermitian operators on~$V_x$.
Let us choose an orthonormal basis of each~$V_x$, and consider the Weyl chamber~$C_x$ of diagonal Hermitian matrices~$\lambda_x$ with nonincreasing real entries $\lambda_x(1)\geq\ldots\geq\lambda_x(n_x)$. When $\lambda_x$ is $\Z$-valued, it determines an irreducible representation $V_{\lambda}$ of $\GL(V_x)$.
Thus, the irreducible representations of $\GL_Q(\mathcal V)$ are of the form $V_{\boldlambda} = \bigotimes_{x\in Q_0} V_{\lambda_x}$, where $\boldlambda = (\lambda_x)_{x\in Q_0}$ is the highest weight.

The cone $C_Q(\CV)$ has an alternative description in terms of symplectic geometry.
Indeed, a moment map for the action of the maximally compact subgroup $U_Q(\mathcal V) = \prod_{x\in Q_0} U(V_x)$ is given by
\begin{align*}
  \mu\colon \CH_Q(\CV) \to \bigoplus_x \sqrt{-1}\u_x, \quad
  v=(v_a)_{a\in Q_1} \mapsto \mu(v) = (\mu_x(v))_{x\in Q_0},
\end{align*}
where~$\mu_x(v)$ is the Hermitian matrix $\sum_{y,b:y\to x} v_b v_b^* - \sum_{y,a : x\to y} v_a^* v_a$.
By the results of Guillemin-Sternberg and Mumford discussed above, an element~$\boldlambda$ of the Weyl chamber~$\prod_x C_x$ is in the cone $C_Q(\CV)$ if and only if $-\boldlambda$ is in the image of the moment map.

We may describe the cone~$C_Q(\CV)$ by an inductively defined set of explicit linear inequalities.
Indeed, a general result by Ressayre~\cite{ressayre2010geometric} (see also~\cite{VW}) implies that~$C_Q(\CV)$ consists of the points~$\boldlambda \in \prod_x C_x$ such that
\begin{align*}
  \sum_{x\in Q_0} \sum_{i=1}^{n_x} \lambda_x(i) = 0
\end{align*}
and, for all $\CK\subset_Q [\boldn]$,
\begin{align*}
  \sum_{x\in Q_0} \sum_{i \in K_x} \lambda_x(i) \leq 0.
\end{align*}
Thus, \cref{eq:refined S} gives a complete and explicit set of linear inequalities for the moment cone~$C_Q(\CV)$.
Following an argument of Ressayre~\cite{ressayrepc}, we also compare~$C_Q(\CV)$ with the cone~$\Sigma_{\tilde Q}(\tilde \CV)$ of weights of semi-invariants for the augmented quiver~$\tilde Q$.
We find that the saturation theorem of Derksen-Weyman~\cite{MR1758750} implies that the conditions above are also sufficient for the irreducible representation~$V_{\boldlambda}$ to appear in~$\Sym^*(\CH_Q(\CV))$, in other words, that the semigroup of highest weights is saturated.
In summary, we obtain the following result (see \cref{sec:rep theo}):

\begin{theorem}\label{thm:moment cone and rep theory summary}
For any highest weight $\boldlambda = (\lambda_x)_{x\in Q_0}$ of $\GL_Q(\CV)$, the following are equivalent:
\begin{enumerate}
\item\label{cond:moment cone} $-\boldlambda$ is in the image of the moment map,
\item\label{cond:highest weight cone} $\boldlambda \in C_Q(\CV)$,
\item $V_\lambda \subseteq \Sym^*(\CH_Q(\CV))$,
\item\label{cond:horn} $\sum\limits_{\crampedclap{x\in Q_0}} \sum_{i=1}^{n_x} \lambda_x(i) = 0$ and $\sum\limits_{\crampedclap{x\in Q_0}} \sum_{i \in K_x} \lambda_x(i) \leq 0$ for all $\CK \subseteq_Q [\boldn]$.
\end{enumerate}
The equivalence between~(\ref{cond:moment cone}), (\ref{cond:highest weight cone}), and~(\ref{cond:horn}) holds also when $\lambda$ is not integral.
Moreover, $\CK \subseteq_Q [\boldn]$ if and only if
\begin{align*}
  \sum\limits_{x\in Q_0} \sum_{j=1}^{l_x} \left( K_x(L_x(j)) - j \right) \;\;\geq \!\!\!\!\sum_{a : x\to y \in Q_1} l_x (n_y - l_y)
\end{align*}
for all $\CL \subseteq_Q [\boldk]$, using the notation of \cref{eq:refined S}.
\end{theorem}

We previously announced this result in~\cite{quivershort}.
Recently, Bertozzi-Reineke~\cite{BR} gave a similar characterization of the image of the moment map based on \cref{thm:ext via eul}, which they proved using augmented quivers.
In \cref{sec:sun quiver}, we give a minimal complete description of~$C_Q(\CV)$ for the `sun quiver'~\cite{Collins} mentioned above.

\subsection{Notation and conventions}
The complement of a subset~$X \subseteq Y$ will be denoted by~$X^c \coloneqq Y \setminus X$.

All vector spaces will be fi\-nite-di\-men\-si\-o\-nal complex vector spaces.
Given a vector space~$V$, we write $\dim V$ for its (complex) dimension, and, for any~$0\leq r\leq\dim V$, we denote by~$\Gr(r,V)$ the Grassmannian that consists of the subspaces of dimension~$r$ of~$V$.

We use calligraphic and bold letters to denote families of objects labeled by the vertex set~$Q_0$ of a quiver.
For example, $\CV=(V_x)_{x\in Q_0}$ will be a family of vector spaces indexed by the set~$Q_0$,
$\CJ=(J_x)_{x\in Q_0}$ a family of subsets~$J_x$ of~$\N=\{1,2,\dots\}$, and~$\boldalpha=(\alpha_x)_{x\in Q_0}$ will be a family of natural numbers.
We write $\Gr_Q(\boldalpha,\CV)$ for the product of Grassmannians $\Gr(\alpha_x, V_x)$, $\dim\CV$ for the vector of dimensions~$\dim V_x$, etc.
The total dimension of~$\CV$ is denoted by~$d(\CV) = \sum_{x\in Q_0} \dim V_x$.
Such families of objects naturally inherit operations and relations.
Thus, given $\boldalpha$ and $\boldbeta$, we write $\boldalpha \leq \boldbeta$ if $\alpha_x \leq \beta_x$ for every~$x\in Q_0$, and we define the maps $\boldalpha\pm\boldbeta$ by $(\boldalpha\pm\boldbeta)_x = \alpha_x\pm\beta_x$.
Similarly, if~$\CS$ and~$\CV$ are families of vector spaces then we write $\CS\subseteq\CV$ if $S_x\subseteq V_x$ for every $x\in Q_0$.
We write~$\CS\subset \CV$ if~$\CS\subseteq\CV$ and~$S_x$ is a proper subspace of~$V_x$ for \emph{at least one} $x\in Q_0$.

\section{Quiver Grassmannians and \texorpdfstring{$Q$}{Q}-intersection}\label{sec:defs}
\begin{definition}[Filtered vector space]\label{def:filtered vector space}
A (complete) \emph{filtration} $F$ on a vector space $V$ is a chain of subspaces
\[ \{0\} = F(0) \subseteq F(1) \subseteq \cdots \subseteq F(i) \subseteq F({i+1}) \subseteq \dots \subseteq F(\ell) = V, \]
such that $\dim F(i+1)\leq\dim F(i)+1$ for all $i=0,\dots,\ell-1$ (i.e., the dimensions increase by at most one in each step).
We call the pair $(V,F)$ a \emph{filtered vector space}.
\end{definition}

The distinct subspaces in a filtration determines a flag.
However, note that the subspaces $F(i)$ need \emph{not} be strictly increasing.
If~$S$ is a subspace of~$V$, then~$S$ inherits the filtration~$F_S(i) \coloneqq F(i)\cap S$, and the quotient space $V/S$ inherits the filtration $F_{V/S}(i)\coloneqq(F(i)+S)/S$.
We will now consider the analogue definitions for families of vector spaces and filtrations.

\begin{definition}[Filtered dimension vector]
Let $\CV = (V_x)_{x\in Q_0}$ be a family of vector spaces.
A \emph{filtration} on~$\CV$ is a family $\CF = (F_x)_{x\in Q_0}$ where each~$F_x$ is a filtration on~$V_x$.
We say that the pair $(\CV,\CF)$ is a \emph{filtered dimension vector}.
\end{definition}

Let $\CS\subseteq\CV$, i.e., $S_x \subseteq V_x$ for every~$x\in Q_0$.
We denote by~$\CV/\CS$ the family of vector spaces~$(V_x/S_x)_{x\in Q_0}$.
If~$\CF$ is a filtration on~$\CV$ then we obtain a filtration~$\CF_\CS$ on~$\CS$ and a filtration~$\CF_{\CV/\CS}$ on the quotient~$\CV/\CS$.

A filtered dimension vector~$(\CV,\CF)$ determines a Borel subgroup of~$\GL_Q(\CV)$, namely~$B_Q(\CV,\CF) = \prod_{x\in Q_0} B_x$, where~$B_x$ is the Borel subgroup of~$\GL(V_x)$ preserving the filtration~$F_x$.
By definition, a \emph{Schubert cell}~$\boldOmega^0 = (\Omega^0_x)_{x\in Q_0}$ is a~$B_Q(\CV,\CF)$-orbit in~$\Gr_Q(\boldalpha,\CV)$.
Its closure $\boldOmega=(\Omega_x)_{x\in Q_0}$ is called a \emph{Schubert variety}.
In other words, each~$\Omega^0_x$ ($\Omega_x$) is a Schubert cell (variety) in~$\Gr(\alpha_x,V_x)$.

We can describe the Schubert varieties more concretely:
Let~$\boldn = (n_x)_{x\in Q_0}$ be a dimension vector.
For~$x\in Q_0$, let~$V_x=\C^{n_x}$, with standard basis~$(e_j)_{1\leq j\leq n_x}$, and consider the standard filtration~$F_x$ corresponding to the Borel subgroup~$B_x$ that consists of the upper-triangular matrices in~$GL(n_x)$.
Let~$\boldalpha$ be a dimension vector such that~$\boldalpha \leq \boldn$.
Let~$\CJ = (J_x)_{x\in Q_0}$ be a family of subsets, where each~$J_x$ is a subset of~$\{1,\dots,n_x\}$ of cardinality~$\alpha_x$.
Then, $S_{J_x} \coloneqq \bigoplus_{j \in J_x} \C e_j$ is a subspace of~$V_x$ of dimension~$\alpha_x$.
Let~$\Omega^0(J_x)$ denote the orbit of~$S_{J_x}$ under the action of~$B_x$, and~$\Omega(J_x)$ its closure.
It is easy to see that
\begin{align*}
  \Omega(J_x) = \{ S \in \Gr(\alpha_x, V_x) : \dim (S \cap F_x(J_x(a))) \geq a \text{ for } 1 \leq a \leq \alpha_x \},
\end{align*}
where $J_x(1) < \dots < J_x(\alpha_x)$ are the elements of~$J_x$.
Then, $\boldOmega(\CJ) = (\Omega(J_x))_{x\in Q_0}$ is a Schubert variety.
Moreover, every Schubert variety in~$\Gr_Q(\boldalpha,\CV)$ is of this form.
\label{eq:concrete schubert}
It is easy to verify that
\begin{align}\label{eq:schubert dim}
  \dim\boldOmega(\CJ) = \sum_{x\in Q_0} \dim\Omega(J_x), \quad \dim\Omega(J_x) = \sum_{a=1}^{\alpha_x} (J_x(a) - a).
\end{align}

\begin{definition}
Let $\CV = (V_x)_{x\in Q_0}$ be a family of vector spaces, $\boldalpha\leq\dim\CV$ a dimension vector, and $v\in\CH_Q(\CV)$ a representation.
Define the corresponding \emph{quiver Grassmannian} as
\begin{align*}
  \Gr_Q(\boldalpha,\CV)_v
\coloneqq \{ \CS \in \Gr_Q(\boldalpha,\CV) : v \CS \subseteq \CS \}.
\end{align*}
We say that $\boldalpha$ is \emph{Schofield subdimension vector} for $\CV$ if $\Gr_Q(\boldalpha,\CV)_v\neq\emptyset$ for every $v\in\CH_Q(\CV)$.
\end{definition}

Quiver Grassmannians have been the subject of intensive research.
We only mention the striking result that, in fact, every projective variety is a quiver Grassmannian~\cite{reineke2013every}.
For particular representations~$v$, cellular decompositions of~$\Gr_Q(\boldalpha,\CV)_v$ have been studied~\cite{QuiverCellular}.

We can decompose each quiver Grassmannians into subvarieties consisting of stable subspaces with fixed Schubert positions.
This gives rise to the central definitions of our article:

\begin{definition}[$Q$-intersecting]\label{def:q-intersecting}
Let $(\CV,\CF)$ be a filtered dimension vector, $\boldalpha\leq\dim\CV$ a dimension vector, and $\boldOmega\subseteq\Gr_Q(\boldalpha,\CV)$ a Schubert variety.
Given a representation $v\in\CH_Q(\CV)$, define
\begin{align*}
  \boldOmega_v
\coloneqq \Gr_Q(\boldalpha,\CV)_v \cap \boldOmega
= \{ \CS \in \boldOmega : v \CS \subseteq \CS \}.
\end{align*}
We say that $\boldOmega$ is \emph{$Q$-intersecting} in~$\CV$ if $\boldOmega_v\neq\emptyset$ for every $v\in\CH_Q(\CV)$.
\end{definition}

In other words, $\boldOmega$ is $Q$-intersecting if, for every $v\in\CH_Q(\CV)$, the Schubert variety~$\boldOmega$ contains a subrepresentation of~$v$.
In this case, we call the variety $\boldOmega_v$ for generic $v$ the \emph{generic intersection variety}.

Clearly, a necessary condition for $\boldOmega$ to be $Q$-intersecting is that $\boldalpha$ is a Schofield subdimension vector.
As we will see in \cref{lem:generic}, $\boldOmega$ is $Q$-intersecting if and only if $\boldOmega_v\neq\emptyset$ for generic~$v\in\CH_Q(\CV)$.

\begin{example}[Horn quiver]\label{ex:horn quiver}
For the Horn quiver~\eqref{eq:horn quiver} and the constant dimension vector $\boldalpha=(r,\dots,r)$, the problem of determining the $Q$-intersection of Schubert varieties in $\Gr_Q(\boldalpha,\CV)$ is equivalent to the problem of determining the intersection of Schubert classes in $\Gr(r,n)$.

Indeed, let $\Omega_1$, \dots, $\Omega_{s+1}$ be Schubert varieties in~$\Gr(r,n)$.
By Kleiman's moving lemma, the homology classes $[\Omega_x]_{x=1}^{s+1}$ are intersecting in $\Gr(r,n)$ if and only if, for every $g_1,\dots,g_{s+1}\in\GL(n)$ there exists a point $S \in \bigcap_{x=1}^{s+1} g_x \Omega_x$.
Define $v_{x\to s+1} \coloneqq g_{s+1}^{-1} g_x$ for $x=1,\dots,s$.
Then $v = (v_{x\to s+1})_{x=1}^s$ is a representation of $H_s$.
Now consider $\boldOmega = (\Omega_1,\dots,\Omega_{s+1})$, which is a Schubert variety in $\Gr_Q(\boldalpha,\CV)$.
Define $S_x = g_x^{-1} S \in \Omega_x$.
Then, $\CS=(S_x)_{x=1}^{s+1} \in \boldOmega$.
Moreover, $v_{x\to s+1} S_x = S_{s+1}$ for $x=1,\dots,s$.
This means that $\CS\in\boldOmega_v$.
The set of $v$ so obtained is dense in~$\CH_Q(\CV)$, since each $v_{x\to s+1}$ can be an arbitrary invertible map~$V_x\to V_{s+1}$.
We conclude that $\boldOmega$ is $H_s$-intersecting if and only if the homology classes $[\Omega_x]_{x=1}^{s+1}$ are intersecting in $\Gr(r,n)$.
\end{example}

Belkale~\cite{MR2177198} has determined an inductive criterion for Schubert classes in~$\Gr(r,n)$ to intersect.
Our aim in this article is to obtain a similar inductive criterion for when a Schubert variety $\boldOmega = (\Omega_x)_{x\in Q_0}$ is $Q$-intersecting.

\section{Expected dimensions}\label{sec:edim}
In this section, we define the expected dimension of the generic intersection variety (\cref{def:edim}).

Given two families of vector spaces $\CV=(V_x)_{x\in Q_0}$ and $\CW=(W_x)_{x\in Q_0}$, define
\begin{align*}
  \CH_Q(\CV,\CW) &\coloneqq \!\!\!\!\!\bigoplus_{a:x\to y \in Q_1}\!\!\!\!\! \Hom(V_x,W_y), \\
  \g_Q(\CV,\CW) &\coloneqq \bigoplus_{x\in Q_0} \Hom(V_x,W_x).
\end{align*}
If $\CV=\CW$, the space $\CH_Q(\CV,\CV)$ is simply $\CH_Q(\CV)$, introduced previously in \cref{eq:quiver repr}, and $\g_Q(\CV,\CV)$ is the Lie algebra~$\gl_Q(\CV)$ of~$\GL_Q(\CV)$.

If $\dim\CV=\boldalpha$ and $\dim\CW=\boldbeta$ then the dimension of $\CH_Q(\CV,\CW)$ is given by
$\sum_{a:x\to y\in Q_1} \alpha_x\beta_y$.
As it depends only on $Q$, $\boldalpha$, and $\boldbeta$, we also denote this expression by $\dim \CH_Q(\boldalpha,\boldbeta)$.
Similarly, the dimension of~$\g_Q(\CV,\CW)$ is
$\sum_{x\in Q_0} \alpha_x\beta_x$.
Thus,
\begin{align*}
  \braket{\boldalpha,\boldbeta} = \dim \g_Q(\CV,\CW)-\dim\CH_Q(\CV,\CW),
\end{align*}
where $\braket{\boldalpha,\boldbeta}$ is the \emph{Euler form} defined in \cref{eq:euler form}.

The following proposition is well known.
We give a proof since we will below generalize it to compute the generic dimension of $\boldOmega_v$.

\begin{proposition}\label{prp:dim generic GrQv}
Let $\CV$ be a family of vector spaces and $\boldalpha$ a Schofield subdimension vector for~$\CV$.
Then, for generic $v\in\CH_Q(\CV)$, the dimension of each irreducible component of $\Gr_Q(\boldalpha,\CV)_v$ is given by
$\dim \Gr_Q(\boldalpha,\CV) - \dim \CH_Q(\boldalpha,\boldbeta) = \braket{\boldalpha,\boldbeta}$,
where~$\boldbeta = \dim\CV - \boldalpha$.
\end{proposition}
\begin{proof}
Define the variety
\begin{align*}
  \mathbb X \coloneqq \{ (\CT, v) \in \Gr_Q(\boldalpha,\CV) \times \CH_Q(\CV) : v \CT \subseteq \CT \}.
\end{align*}
The map
\begin{align*}
  p\colon \mathbb X \to \Gr_Q(\boldalpha,\CV), \quad (\CT, v) \mapsto \CT
\end{align*}
equips $\mathbb X$ with the structure of a vector bundle over $\Gr_Q(\boldalpha,\CV)$.
Indeed, let $\CT\in\Gr_Q(\boldalpha,\CV)$.
We can write $\CV=\CT\oplus\CU$, choosing for each $x\in Q_0$ a complement~$U_x$ of~$S_x$ in~$V_x$.
Thus, $\dim(\CU)=\boldbeta$.
The fiber $p^{-1}(\CT)$ can be identified with
\begin{align}\label{eq:fiber X(S)}
  \mathbb X(\CT) = \{ v \in \CH_Q(\CV) : v \CT \subseteq \CT \}.
\end{align}
The right-hand side condition means that $v$ is of the form
\begin{align*}
  v = \begin{pmatrix}v_{00} & v_{01} \\ 0 & v_{11}\end{pmatrix},
\end{align*}
where $v_{00} \in \CH_Q(\CT)$, $v_{01} \in \CH_Q(\CU, \CT)$, and $v_{11} \in \CH_Q(\CU)$.
Thus,~$\mathbb X(\CT)$ is a vector subspace of $\CH_Q(\CV)$ of codimension $\dim \CH_Q(\CT, \CU) = \dim \CH_Q(\boldalpha, \boldbeta)$.
It follows that $\mathbb X$ is irreducible and of dimension
\begin{align}\label{eq:dim X}
  \dim\mathbb X = \dim \Gr_Q(\boldalpha,\CV) + \dim \CH_Q(\CV) - \dim \CH_Q(\boldalpha,\boldbeta).
\end{align}
We also have a map
\begin{align*}
  q\colon \mathbb X \to \CH_Q(\CV), \quad (\CT, v) \mapsto v,
\end{align*}
whose fibers can be identified with~$\Gr_Q(\boldalpha,\CV)_v$.
If $\boldalpha$ is a Schofield subdimension vector then the map~$q$ is surjective.
By the version of Sard's theorem for dominant maps between irreducible varieties, it follows that the image of~$q$ contains a nonempty Zariski-open subset~$Z \subseteq \CH_Q(\CV)$ such that, for $v\in Z$, each irreducible component of the fiber $\Gr_Q(\boldalpha,\CV)_v$ is of dimension equal to $\dim\mathbb X - \dim \CH_Q(\CV)$.
Comparing with \cref{eq:dim X}, we obtain that, for generic $v$, each irreducible component of $\Gr_Q(\boldalpha,\CV)_v$ is of dimension
\begin{align*}
  \dim\mathbb X - \dim \CH_Q(\CV)
= \dim \Gr_Q(\boldalpha,\CV) - \dim \CH_Q(\boldalpha,\boldbeta)
= \braket{\boldalpha, \boldbeta}.
\end{align*}
In the last step, we used that $\dim \Gr(\alpha_x, V_x) = \alpha_x \beta_x$ for~$x\in Q_0$.
\end{proof}

In particular, we see that a necessary condition for $\boldalpha$ to be a Schofield subdimension vector is that $\braket{\boldalpha,\boldbeta}\geq0$, where $\boldbeta=\dim\CV-\boldalpha$.
We now prove an analog of \cref{prp:dim generic GrQv} for generic intersection varieties.

\begin{proposition}\label{prp:dim generic Omegav}
Let $(\CV,\CF)$ be a filtered dimension vector, $\boldalpha\leq\dim\CV$ a dimension vector, and $\boldOmega\subseteq\Gr_Q(\boldalpha,\CV)$ a $Q$-intersecting Schubert variety.
Then, for generic $v\in\CH_Q(\CV)$, the dimension of each irreducible component of $\boldOmega_v$ is given by $\dim\boldOmega - \dim\CH_Q(\boldalpha,\boldbeta)$, where $\boldbeta=\dim\CV-\boldalpha$.
\end{proposition}
\begin{proof}
The proof is entirely similar.
This time, we  consider
\begin{align}\label{eq:X prime}
  \mathbb X \coloneqq \{ (\CT,v) \in \boldOmega \times \CH_Q(\CV) : v \CT \subseteq \CT \},
\end{align}
which has now the structure of a vector bundle over $\boldOmega$, with fibers as in \cref{eq:fiber X(S)}.
Similarly to \cref{eq:dim X}, it follows that $\mathbb X$ is an irreducible variety of dimension
\begin{align*}
  \dim \mathbb X = \dim\boldOmega + \dim\CH_Q(\CV) - \dim\CH_Q(\boldalpha,\boldbeta).
\end{align*}
If $\boldOmega$ is $Q$-intersecting, the map
\begin{align}\label{eq:q map}
  q\colon \mathbb X \to \CH_Q(\CV), \quad (\CT, v) \mapsto v,
\end{align}
is surjective.
As its fibers can be identified with $\boldOmega_v$, we conclude as before that the dimension of each irreducible component is, for generic $v$, given by $\dim\boldOmega - \dim\CH_Q(\boldalpha,\boldbeta)$.
\end{proof}

Thus, we find that a necessary condition for $\boldOmega\subseteq\Gr_Q(\boldalpha,\CV)$ to be $Q$-intersecting is that $\dim\boldOmega - \dim\CH_Q(\boldalpha,\boldbeta) \geq 0$, where $\boldbeta=\dim\CV-\boldalpha$.
Using \cref{eq:schubert dim}, the latter condition is easy to evaluate for a Schubert variety~$\boldOmega(\CJ)$.
It amounts to
\begin{align*}
  \sum_{x\in Q_0} \sum_{a=1}^{\alpha_x} (J_x(a) - a) - \sum_{a:x\to y} \alpha_x (n_y - \alpha_y) \geq 0.
\end{align*}

Next, we study Schubert cells and varieties determined by families of subspaces.

\begin{definition}[$Q$-intersecting families of subspaces]
Let $(\CV,\CF)$ be a filtered dimension vector, $\boldalpha\leq\dim\CV$ a dimension vector, and~$\CS\in\Gr_Q(\boldalpha,\CV)$.
We define~$\boldOmega^0(\CS,\CF)$ as the~$B_Q(\CV,\CF)$-orbit of~$\CS$, and denote by~$\boldOmega(\CS,\CF)$ its closure, which is a Schubert variety.

We say that~$\CS$ is \emph{$Q$-intersecting} in~$\CV$ if~$\boldOmega(\CS,\CF)$ is $Q$-intersecting in the sense of \cref{def:q-intersecting} and denote this condition by~$\CS \subseteq_Q \CV$.
We write~$\CS \subset_Q \CV$ if in addition at least one subspace is a proper subspace.
\end{definition}

The following lemma is similar to \cite{BVW}*{Lemma~4.2.4}.

\begin{lemma}\label{lem:generic}
Let $(\CV,\CF)$ be a filtered dimension vector and $\CS\subseteq\CV$ a family of subspaces.
If $\CS$ is $Q$-intersecting in~$\CV$, there exists a nonempty Zariski-open set of $v\in\CH_Q(\CV)$ such that $\boldOmega^0(\CS,\CF)$ contains a subrepresentation of~$v$.

Conversely, if $\boldOmega^0(\CS,\CF)$ contains a subrepresentation of~$v$ for generic~$v\in\CH_Q(\CV)$, then $\CS$ is $Q$-intersecting in~$\CV$.
\end{lemma}
\begin{proof}
Abbreviate $\boldOmega=\boldOmega(\CS,\CF)$ and $\boldOmega^0=\boldOmega^0(\CS,\CF)$.
Consider the manifold
\begin{align}\label{eq:X0}
  \mathbb X^0 \coloneqq \{ (\CT,v) \in \boldOmega^0 \times \CH_Q(\CV) : v \CT \subseteq \CT \},
\end{align}
which is a nonempty Zariski-open subset of the irreducible variety~$\mathbb X$ defined in \cref{eq:X prime}.
If $\boldOmega$ is $Q$-intersecting, the map~$q$ defined in \cref{eq:q map} is surjective.
Thus, it is also dominant on any nonempty Zariski-open subset of~$\mathbb X$, hence in particular on $\mathbb X^0$.
It follows that the image of \cref{eq:q map} contains a nonempty Zariski-open subset of representations~$v\in\CH_Q(\CV)$ with the property that $\boldOmega^0$ contains a subrepresentation of~$v$.

Conversely, suppose that $\boldOmega^0$ contains a subrepresentation of~$v$ for generic $v\in\CH_Q(\CV)$.
Then, since the closure $\boldOmega$ of $\boldOmega^0$ is compact, it follows that $\boldOmega$ contains subrepresentations of all~$v\in\CH_Q(\CV)$.
\end{proof}

We now define the expected dimension as the expression in \cref{prp:dim generic Omegav}.

\begin{definition}[Expected dimension]\label{def:edim}
Let $(\CV,\CF)$ be a filtered dimension vector and $\CS\subseteq \CV$ a family of subspaces.
We define
\begin{align*}
  \edim_{Q,\CF}(\CS,\CV) \coloneqq \dim\boldOmega(\CS,\CF) - \dim\CH_Q(\CS,\CV/\CS)
\end{align*}
and call it the \emph{expected dimension} of the intersection variety~$\boldOmega(\CS,\CF)_v$.
\end{definition}

\noindent
Thus, the following lemma is clear.

\begin{lemma}\label{lem:necessary}
Let $(\CV,\CF)$ be a filtered dimension vector and $\CS\subseteq \CV$ a family of subspaces.
If $\CS$ is $Q$-intersecting in~$\CV$, then $\edim_{Q,\CF}(\CS,\CV)\geq0$.
\end{lemma}

The converse of \cref{lem:necessary} is not in general true.
That is, it is possible that~$\edim_{Q,\CF}(\CS,\CV)\geq0$ even when~$\CS$ is not $Q$-intersecting.
We already saw an example of this when discussing the quiver~\eqref{eq:square quiver} in \cref{sec:intro}.


If $\CS$ is $Q$-intersecting and $\edim_{Q,\CF}(\CS,\CV)=0$, this means that the generic intersection variety $\boldOmega(\CS,\CF)_v$ is a finite set of points.
We now consider the important special case when it is a \emph{single} point.

\begin{definition}\label{def:point}
Let $(\CV,\CF)$ be a filtered dimension vector.
We define~$P_Q(\CV,\CF)$ as the set of subspaces $\CS\subseteq \CV$ such that, for generic $v\in\CH_Q(\CV)$, the intersection variety~$\boldOmega(\CS,\CF)_v$ is equal to a point.
\end{definition}

If $\CS\in P_Q(\CV,\CF)$ then $\CS$ is $Q$-intersecting in $\CV$ and $\edim_{Q,\CF}(\CS,\CV)=0$.
But the converse is not usually true, as the following example shows.

\begin{example}\label{ex:quiver W2}
Let $W_2$ be the following quiver:
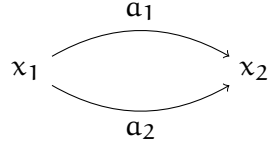
\begin{equation*}
\begin{aligned}
\begin{tikzpicture}\small
  \node (A) at (-1.5,0) {$x_1$};
  \node (B) at (1.5,0) {$x_2$};
  \draw[->] (A) edge[bend left] node[above] {$a_1$} (B);
  \draw[->] (A) edge[bend right] node[below] {$a_2$} (B);
\end{tikzpicture}
\end{aligned}
\end{equation*}
Let $\CV=(\C^2,\C^2)$, $\CF$ the standard filtration, and consider $\CS=(\C e_2, \C e_2)$.
Then, $\boldOmega(\CS,\CF) = \Gr(1,2) \times \Gr(1,2)$ has dimension~$2$, and
\begin{align*}
  \edim_{Q,\CF}(\CS,\CV) = 2 - (1 + 1) = 0.
\end{align*}
Now let $v=(v_1,v_2)\in\CH_{W_2}(\CV) = \Hom(\C^2,\C^2) \oplus \Hom(\C^2,\C^2)$.
For generic~$v$, both~$v_1$ and~$v_2$ are invertible.
If $L$ is an eigenvector of~$v_2^{-1} v_1$, then we have $v_1(\C L) = v_2(\C L)$, which implies that $(\C L, v_1(\C L))$ is a subrepresentation of~$v$, and trivially contained in $\boldOmega(\CS,\CF)$.
Thus, $\CS$ is also $Q$-in\-ter\-sec\-ting.
However, $v_2^{-1} v_1$ is generically diagonalizable, in which case there are \emph{two} such subrepresentations of~$v$.
Thus, $\CS$ is not in $P_{W_2}(\CV,\CF)$.
\end{example}

Derksen-Schofield-Weyman~\cite{DSW} have determined the number of subrepresentations of a general quiver representation in terms of certain multiplicities.

The following lemma shows that the notion of $Q$-intersection is transitive.

\begin{lemma}\label{lem:transitive}
Let $(\CV,\CF)$ be a filtered dimension vector and $\CT\subseteq\CS\subseteq\CV$ families of subspaces.
Assume that $\CS \subseteq_Q \CV$ and $\CT \subseteq_Q \CS$, where $\CS$ is equipped with the filtration $\CF_\CS$.
Then, $\CT \subseteq_Q \CV$.
\end{lemma}
\begin{proof}
Let $v\in\CH_Q(\CV)$ be generic.
Since $\CS\subseteq_Q \CV$, \cref{lem:generic} shows that there exists $b\in B_Q(\CV,\CF)$ such that $\tilde v=bvb^{-1}$ satisfies $\tilde v \CS \subseteq \CS$.

Since $\CT \subseteq_Q \CS$, there exists $\CN \in \boldOmega(\CT,\CF_\CS)$ such that~$\tilde v\CN \subseteq \CN$.
Every element $g\in B_Q(\CS,\CF_\CS)$ is the restriction of an element~$h\in B_Q(\CV,\CF)$ with $h \CS=\CS$.
It follows that $\boldOmega(\CT,\CF_\CS)$ is contained in $\boldOmega(\CT,\CF)$, hence~$\CN \in \boldOmega(\CT,\CF)$.
It follows that $v(b^{-1} \CN) \subseteq b^{-1} \CN$.
Since $b^{-1} \CN$ still belongs to $\boldOmega(\CT,\CF)$, we see that $\CT \subseteq_Q \CV$.
\end{proof}

\Cref{lem:necessary,lem:transitive} show that the two conditions~\ref{it:main A} and~\ref{it:main B} in \cref{thm:main} are necessary for $\CS$ to be $Q$-intersecting in~$\CV$.

The objective of the following sections is to prove the converse statement.
In fact, we will prove a refinement of \cref{thm:main}:
In \cref{thm:main refined}, we will show that in condition~\ref{it:main B} it suffices to consider only those~$\CT\neq\CS$ such that $\CT\in P_Q(\CS,\CF_\CS)$.
In turn, we obtain simple Horn conditions for testing $Q$-intersection (\cref{sec:horn}).
In the case of the Horn quivers, these conditions can be readily reduced to Belkale's conditions for intersecting Schubert classes~\cite{MR2177198}.
This emblematic example suggested to us the statement of the more general theorem.

\section{\texorpdfstring{$\Ext$}{Ext} groups and Schofield Criterium}
The proof of \cref{thm:main} will be based on computing the dimension of an $\Ext$ group.
We first state some easy lemmas about filtered vector spaces with proofs left to the reader.
Given two filtered vector spaces~$(V,F)$ and~$(W,G)$, a homomorphism~$\Phi\colon V\to W$ is a linear map that respect the two filtrations, i.e., $\Phi(F(i)) \subseteq G(i)$ for all~$i$ (we assume that both filtrations have the same length).
We denote the space of morphisms by~$\g_{F,G}(V,W)$.

\begin{lemma}
Let $(V,F)$ be a filtered vector space and $S\subseteq V$ a subspace.
Then, the exact sequence
  $0\to (S,F_S) \to (V,F) \to (V/S, F_{V/S}) \to 0$
is split.
\end{lemma}
%

\begin{lemma}\label{lem:dim filtered morphisms}
Let $(V,F)$ and $(W,G)$ be filtered vector spaces and $r=\dim V$.
Let $i_1 < \dots < i_r$ denote the smallest indices such that $\dim F(i_a) = a$ for~$a=1,\dots,r$.
Then, $\dim \g_{F,G}(V,W) = \sum_{a=1}^r \dim G(i_a)$.
\end{lemma}

Let $B(V,F) \subseteq \GL(V)$ be the Borel subgroup associated to~$F$.
Its Lie algebra is $\b(V,F) = \g_{F,F}(V,V) \subseteq \gl(V)$.
It is clear that any $X\in \b(V,F)$ induces a map $\Phi \in \g_{F_S,F_{V/S}}(S, V/S)$.


\begin{lemma}\label{lem:surjective S to V/S}
The map $\b(V,F) \to \g_{F_S,F_{V/S}}(S, V/S)$ is surjective.
\end{lemma}

Finally, we record the following lemma:

\begin{lemma}\label{lem:add dim}
Let $(V,F)$ and $(W,G)$ be filtered vector spaces and let $S\subseteq V$ and $T\subseteq W$ be subspaces.
Then:
\begin{align*}
  \dim \g_{F,G}(V,W) &= \dim \g_{F_S,G}(S,W) + \dim \g_{F_{V/S},G}(V/S,W), \\
  \dim \g_{F,G}(V,W) &= \dim \g_{F,G_T}(V,T) + \dim \g_{F,G_{W/T}}(V,W/T).
\end{align*}
\end{lemma}

We now consider families of filtered vector spaces, i.e., filtered dimension vectors.
Given two filtered dimension vectors $(\CV,\CF)$ and $(\CW,\CG)$, a homomorphism $\Phi=(\Phi_x)_{x\in Q_0}$ consists of a family of maps~$\Phi_x \in \g_{F_x,G_x}(V_x,W_x)$.
We denote the space of homomorphisms by~$\g_{Q,\CF,\CG}(\CV,\CW)$.
As above, $\b_Q(\CF,\CV) = \g_{Q,\CF,\CF}(\CV,\CV) \subseteq \gl_Q(\CV)$ is the Lie algebra of a Borel subgroup of~$\GL_Q(\CV)$.
The following definition is the filtered analog of \cref{eq:euler form}.

\begin{definition}[Filtered Euler number]\label{def:filtered eul}
Let $(\CV,\CF)$ and $(\CW,\CG)$ be two filtered dimension vectors.
We define the \emph{filtered Euler number} by
\begin{align*}
  \eul_{Q,\CF,\CG}(\CV,\CW) \coloneqq \dim \g_{Q,\CF,\CG}(\CV,\CW) - \dim\CH_Q(\CV,\CW).
\end{align*}
\end{definition}

For families of subspaces~$\CS\subseteq\CV$ and $\CT\subseteq\CW$, \cref{lem:add dim} implies that
\begin{align}
\label{eq:add eul}
  \eul_{Q,\CF,\CG}(\CV,\CW) &= \eul_{Q,\CF_\CS,\CG}(\CS,\CW) + \eul_{Q,\CF_{\CV/\CS},\CG}(\CV/\CS,\CW), \\
\label{eq:add eul 2nd}
  \eul_{Q,\CF,\CG}(\CV,\CW) &= \eul_{Q,\CF,\CG_\CT}(\CV,\CT) + \eul_{Q,\CF,\CG_{\CW/\CT}}(\CV,\CW/\CT).
\end{align}

Filtered Euler numbers can be computed in the following way.
For $v=(v_a)_{a\in Q_1}\in\CH_Q(\CV)$ and $w=(w_a)_{a\in Q_1}\in\CH_Q(\CW)$, consider the map
\begin{align}\label{eq:delta}
  \delta_{v,w}\colon \g_{Q,\CF,\CG}(\CV,\CW) \to \CH_Q(\CV,\CW), \; \Phi \mapsto \Phi v - w \Phi,
\end{align}
where the right-hand side denotes the element of $\CH_Q(\CV,\CW)$ with components~$\Phi_y v_a - w_a \Phi_x$ for each arrow $a:x\to y$ in $Q_1$, generalizing our notation for the action of $\gl_Q(\CV)$ on $\CH_Q(\CV)$.
Define
\begin{align*}
  \Hom_{Q,\CF,\CG}(v,w) &\coloneqq \ker(\delta_{v,w}),\\
  \Ext_{Q,\CF,\CG}(v,w) &\coloneqq \coker(\delta_{v,w}),
\end{align*}
so that we have a short exact sequence
\begin{align*}
  0 \!\!\to\!\! \Hom_{Q,\CF,\CG}(v,w) \!\!\to\!\! \g_{Q,\CF,\CG}(\CV,\CW) \!\!\to\!\! \CH_Q(\CV,\CW) \!\!\to\!\! \Ext_{Q,\CF,\CG}(v,w) \!\!\to\!\! 0,
\end{align*}
By exactness, the Euler number of this complex is zero, hence
\begin{align*}
  \eul_{Q,\CF,\CG}(\CV,\CW) = \dim \Hom_{Q,\CF,\CG}(v,w)-\dim \Ext_{Q,\CF,\CG}(v,w)
\end{align*}
for any~$v\in\CH_Q(\CV)$ and $w\in\CH_Q(\CW)$.
Now define
\begin{align*}
  \hom_{Q,\CF,\CG}(\CV,\CW)&\coloneqq\min_{v,w} \dim \Hom_{Q,\CF,\CG}(v,w),\\
  \ext_{Q,\CF,\CG}(\CV,\CW)&\coloneqq\min_{v,w} \dim \Ext_{Q,\CF,\CG}(v,w),
\end{align*}
where the minimizations are over all $v\in\CH_Q(\CV)$ and $w\in\CH_Q(\CW)$.
There exists a Zariski-open subset where both minima are simultaneously attained, hence
\begin{align}\label{eq:eul = hom - ext}
  \eul_{Q,\CF,\CG}(\CV,\CW) = \hom_{Q,\CF,\CG}(\CV,\CW)-\ext_{Q,\CF,\CG}(\CV,\CW).
\end{align}

If $\CS\subseteq\CV$ is a family of subspaces then the tangent space at~$\CS$ of the Schubert cell~$\boldOmega^0(\CS,\CF)$ can be identified with~$\g_{Q,\CF_\CS,\CF_{\CV/\CS}}(\CS,\CV/\CS)$.
Thus:
\begin{align*}
  \dim \boldOmega(\CS,\CF) = \dim \g_{Q,\CF_\CS,\CF_{\CV/\CS}}(\CS,\CV/\CS),
\end{align*}
hence, using \cref{def:edim,def:filtered eul},
\begin{equation}\label{eq:eul is edim}
\begin{aligned}
  \eul_{Q,\CF_\CS,\CF_{\CV/\CS}}(\CS,\CV/\CS)
&= \dim \boldOmega(\CS,\CF) - \dim \CH_Q(\CS,\CV/\CS) \\
&= \edim_{Q,\CF}(\CS,\CV).
\end{aligned}
\end{equation}
Our next theorem is the analog of Schofield's theorem~\cite{MR1162487} in the context of filtered dimension vectors:

\begin{theorem}\label{thm:Q-intersecting iff ext=0}
Let $(\CV,\CF)$ be a filtered dimension vector and $\CS\subseteq\CV$ a family of subspaces.
Then $\CS\subseteq_Q \CV$ if and only if $\ext_{Q,\CF_{\CS},\CF_{\CV/\CS}}(\CS,\CV/\CS)=0$.
\end{theorem}
\begin{proof}
Abbreviate $\boldOmega^0=\boldOmega^0(\CS,\CF)$.
Consider again the smooth variety from \cref{eq:X0},
\begin{align*}
  \mathbb X^0 = \{ (\CT,v) \in \boldOmega^0 \times \CH_Q(\CV) : v \CT \subseteq \CT \},
\end{align*}
which is a $B_Q(\CV,\CF)$-equivariant vector bundle over the homogeneous space~$\boldOmega^0$.
Recall from \cref{eq:fiber X(S)} that the fiber $\mathbb X(\CS)$ is the vector space consisting of all elements
\begin{align}\label{eq:elem of fiber}
   v = \begin{pmatrix}v_{00} & v_{01} \\ 0 & v_{11}\end{pmatrix}
\end{align}
with $v_{00} \in \CH_Q(\CS)$, $v_{01} \in \CH_Q(\CU, \CS)$, and $v_{11} \in \CH_Q(\CU)$, where $\CU$ is a complement of~$\CS$ in~$\CV$.
Now consider the map
\begin{align*}
  m \colon B_Q(\CV,\CF) \times \mathbb X(\CS) \to \CH_Q(\CV), \quad (b,v) \mapsto b v b^{-1}.
\end{align*}
Then, $\CS\subseteq_Q\CV$ if and only if the map~$m$ is dominant.
Since $m$ is a map between smooth irreducible varieties, it is dominant if and only if there exists a point $(b,v)$ where the differential is surjective.
By equivariance, we can assume that $b=1$.
Thus, $\CS\subseteq_Q\CV$ if and only if the differential of~$m$ at $(1,v)$ is surjective for some~$v$.

This differential can be written as
\begin{align*}
  \b_Q(\CV,\CF) \oplus \mathbb X(\CS) \to \CH_Q(\CV), \quad (X,w) \mapsto Xv - vX + w,
\end{align*}
where $X \in \b_Q(\CV,\CF)$ and $w \in \mathbb X(\CS)$.
In view of \cref{eq:elem of fiber}, this map is surjective if and only if its `component' $\b_Q(\CV,\CF) \to \CH_Q(\CS, \CU) \cong \CH_Q(\CS, \CV/\CS)$ is surjective.
Since $\b_Q(\CV,\CF)$ surjects onto $\g_{Q,\CF_\CS,\CF_{\CV/\CS}}(\CS,\CV/\CS)$ by \cref{lem:surjective S to V/S}, it even suffices to determine when
\begin{align*}
  \g_{Q,\CF_\CS,\CF_{\CV/\CS}}(\CS,\CV/\CS) \to \CH_Q(\CS, \CV/\CS), \quad \Phi \mapsto \Phi v_{00} - v_{11} \Phi
\end{align*}
is surjective.
But this is exactly the map~$\delta_{v_{00},v_{11}}$ from \cref{eq:delta}.
Thus, we conclude that $\CS\subseteq_Q\CV$ if and only if
$\ext_{Q,\CF_\CS,\CF_{\CV/\CS}}(\CS,\CV/\CS)=0$.
\end{proof}

\section{Calculation of ext}\label{sec:ext}
Let $(\CV,\CF)$ and $(\CW,\CG)$ be filtered dimension vectors.
In this section, we compute the quantity~$\ext_{Q,\CF,\CG}(\CV,\CW)$ in terms of a minimization over filtered Euler numbers (\cref{def:filtered eul}).
Using \cref{thm:Q-intersecting iff ext=0}, this reduces the problem of determining $Q$-intersection to an easy numerical criterion.

\begin{theorem}\label{thm:ext via eul}
Let $(\CV,\CF)$ and $(\CW,\CG)$ be filtered dimension vectors.
Then,
\begin{align*}
  \ext_{Q,\CF,\CG}(\CV,\CW) = -\min_{\CS\subseteq_Q \CV} \eul_{Q,\CF_{\CS},\CG}(\CS,\CW),
\end{align*}
where we minimize over all $\CS\subseteq_Q\CV$ including $\CS=(\{0\})$ and $\CS=\CV$.
\end{theorem}

The minimization is well-defined, since $\eul_{Q,\CF_{\CS},\CG}(\CS,\CW)$ only depends on the $B_Q(\CV,\CF)$-orbit of~$\CS$ (i.e., the Schubert cell determined by~$\CS$) and there are only finitely many such orbits.
The remainder of this section will be concerned with the proof of \cref{thm:ext via eul}.

Let $v\in\CH_Q(\CV)$, $w\in\CH_Q(\CW)$, and $\CS\subseteq\CV$ a subrepresentation of~$v$.
Consider the surjective map
\begin{align}\label{eq:res and proj}
  \CH_Q(\CV,\CW) \to \CH_Q(\CS,\CW) \to \Ext_{Q,\CF_\CS,\CG}(v|_\CS,w)
\end{align}
where the first arrow is componentwise restriction and the second the canonical quotient map.
The proof of the following lemma is left to the reader.

\begin{lemma}\label{lem:ext ieq v w}
The map~\eqref{eq:res and proj} descends to a surjection
\begin{align*}
  \Ext_{Q,\CF,\CG}(v,w) \to \Ext_{Q,\CF_\CS,\CG}(v|_\CS,w).
\end{align*}
In particular, for any two representations $v\in\CH_Q(\CV)$ and $w\in\CH_Q(\CW)$ we have that~$\dim \Ext_{Q,\CF,\CG}(v,w) \geq \dim \Ext_{Q,\CF_\CS,\CG}(v|_\CS,w)$.
\end{lemma}

\begin{lemma}\label{lem:ext ieq}
Let $\CS\subseteq_Q\CV$.
Then,
$\ext_{Q,\CF,\CG}(\CV,\CW) \geq \ext_{Q,\CF_\CS,\CG}(\CS,\CW)$.
\end{lemma}
\begin{proof}
  For generic $v\in\CH_Q(\CV)$ and $w\in\CH_Q(\CW)$,
  \begin{align*}
    \dim \Ext_{Q,\CF,\CG}(v,w) = \ext_{Q,\CF,\CG}(\CV,\CW)
  \end{align*}
  and $v$ has a subrepresentation~$\CT$ in the~$B_Q(\CV,\CF)$-orbit of~$\CS$ (since $\CS$ is $Q$-intersecting).
  Thus:
  \begin{align*}
    \ext_{Q,\CF,\CG}(\CV,\CW)
  &= \dim \Ext_{Q,\CF,\CG}(v,w)
  \geq \dim \Ext_{Q,\CF_\CT,\CG}(v|_\CT,w)\\
  &\geq \ext_{Q,\CF_\CT,\CG}(\CT,\CW)
  = \ext_{Q,\CF_\CS,\CG}(\CS,\CW).
  \end{align*}
  The first inequality is \cref{lem:ext ieq v w}.
  The equality at the end holds by~$B_Q(\CV,\CF)$-invariance.
\end{proof}

\begin{proof}[Proof of \cref{thm:ext via eul}]
It follows from \cref{lem:ext ieq,eq:eul = hom - ext} that, for every $\CS\subseteq_Q\CV$,
\begin{align}\label{eq:ext via eul ieq}
  \ext_{Q,\CF,\CG}(\CV,\CW) \geq -\eul_{Q,\CF_\CS,\CG}(\CS,\CW).
\end{align}

We will prove by induction over the dimension of $\CV$ that there always exists~$\CS\subseteq_Q\CV$ that saturates the inequality.
If $\hom_{Q,\CF,\CG}(\CV,\CW)=0$ then \cref{eq:eul = hom - ext} shows that equality holds for~$\CS=\CV$.
This also covers the base case of the induction (i.e., the case that $d(\CV)=0$).
We can therefore assume that $\hom_{Q,\CF,\CG}(\CV,\CW)>0$.
Consider:
\begin{align*}
  \mathbb Y \coloneqq \{ (\Phi,v,w) : \Phi\in\Hom_{Q,\CF,\CG}(v,w), v\in\CH_Q(\CV), w\in\CH_Q(\CW) \}
\end{align*}
(\Cref{ex:quiver W2 red} below shows that $\mathbb Y$ need not be irreducible.)
Consider the projection
\begin{align*}
  q&\colon \mathbb Y \to \CH_Q(\CV) \times \CH_Q(\CW), \; (\Phi,v,w) \mapsto (v,w).
\end{align*}
Let $Z$ denote the nonempty Zariski-open subset of $(v,w)\in\CH_Q(\CV)\times\CH_Q(\CW)$ where $\dim \Hom_{Q,\CF,\CG}(v,w)=\hom_{Q,\CF,\CG}(\CV,\CW)$.
Then, $\mathbb Y_q \coloneqq q^{-1}(Z)$ is a vector bundle over $Z$ with fiber of dimension~$\hom_{Q,\CF,\CG}(\CV,\CW)$.
Since $Z$ is Zariski-open, it follows that $\mathbb Y_q$ is a smooth irreducible variety of dimension
\begin{equation}\label{eq:dim Y_q}
\begin{aligned}
  \dim \mathbb Y_q
&= \dim Z + \hom_{Q,\CF,\CG}(\CV,\CW)\\
&= \dim \CH_Q(\CV) + \dim \CH_Q(\CW) + \hom_{Q,\CF,\CG}(\CV,\CW).
\end{aligned}
\end{equation}

For each $x\in Q_0$, let $\delta_x$ denote the minimal dimension of $\ker(\Phi_x)$ as we vary $(\Phi,v,w)\in\mathbb Y_q$.
There exists a nonempty Zariski-open subset of $\mathbb Y_q$ where the minimum is obtained for every $x\in Q_0$.
It follows that $\bolddelta=(\delta_x)_{x\in Q_0}$ is the dimension vector of a family of subspaces~$\ker(\Phi)\subseteq\CV$.

In fact, $\bolddelta$ is a Schofield subdimension vector.
Indeed, by construction, for generic~$v$ there exists $(w,\Phi)$ such that $(v,w)\in Z$, $\Phi\in\Hom_{Q,\CF,\CG}(v,w)$, and $\dim\ker\Phi=\bolddelta$.
The condition $\Phi v = w\Phi$ implies that~$\ker(\Phi)$ is a subrepresentation of~$v$.
Moreover, $\bolddelta\neq\dim\CV$, since $\hom_{Q,\CF,\CG}(\CV,\CW)>0$ by assumption.

We can further consider the subspaces $\ker(\Phi_x)\cap F_x(i)$ for each~$x\in Q_0$ and~$i$ and similarly minimize their dimensions.
We thus obtain a Zariski-open subset of $\mathbb Y_q$ such that $\ker(\Phi)$ belongs to a fixed Schubert cell $\boldOmega^0(\CS,\CF)$ of $\Gr_Q(\bolddelta,\CV)$.
We call $\CS$ a \emph{generic kernel subrepresentation}.
Note that $\CS\subset_Q\CV$, arguing as before.

\begin{claim}\label{claim:dim hom via eul}
$\hom_{Q,\CF,\CG}(\CV,\CW) = \eul_{Q,\CF_\CS,\CF_{\CV/\CS}}(\CS,\CV/\CS) + \eul_{Q,\CF_{\CV/\CS},\CG}(\CV/\CS,\CW)$.
\end{claim}

\begin{claim}\label{claim:hom ieq}
$\hom_{Q,\CF_\CS,\CG}(\CS,\CW) \geq \hom_{Q,\CF_\CS,\CF_{\CV/\CS}}(\CS,\CV/\CS)$.
\end{claim}

\noindent
We will prove these two claims below.
As a consequence,
\begin{align*}
&\quad \ext_{Q,\CF,\CG}(\CV,\CW) - \ext_{Q,\CF_\CS,\CG}(\CS,\CW) \\
&= \hom_{Q,\CF,\CG}(\CV,\CW) - \eul_{Q,\CF,\CG}(\CV,\CW) - \hom_{Q,\CF_\CS,\CG}(\CS,\CW) + \eul_{Q,\CF_\CS,\CG}(\CS,\CW) \\
&= \hom_{Q,\CF,\CG}(\CV,\CW) - \eul_{Q,\CF_{\CV/\CS},\CG}(\CV/\CS,\CW) - \hom_{Q,\CF_\CS,\CG}(\CS,\CW) \\
&= \eul_{Q,\CF_\CS,\CF_{\CV/\CS}}(\CS,\CV/\CS) - \hom_{Q,\CF_\CS,\CG}(\CS,\CW) \\
&\leq \eul_{Q,\CF_\CS,\CF_{\CV/\CS}}(\CS,\CV/\CS) - \hom_{Q,\CF_\CS,\CF_{\CV/\CS}}(\CS,\CV/\CS) \\
&= -\ext_{Q,\CF_\CS,\CF_{\CV/\CS}}(\CS,\CV/\CS)
\leq 0
\end{align*}
Here we used \cref{eq:eul = hom - ext}, \cref{eq:add eul}, \cref{claim:dim hom via eul}, \cref{claim:hom ieq}, and again \cref{eq:eul = hom - ext} (in this order).
Thus, we obtain that $\ext_{Q,\CF,\CG}(\CV,\CW) \leq \ext_{Q,\CF_\CS,\CG}(\CS,\CW)$.
Since the reverse inequality also holds by \cref{lem:ext ieq}, we obtain the following fundamental formula:
\begin{align}\label{eq:ext=ext}
  \ext_{Q,\CF,\CG}(\CV,\CW) = \ext_{Q,\CF_\CS,\CG}(\CS,\CW).
\end{align}

This readily allows us to conclude the proof of the theorem.
Since~$\CS\subset_Q\CV$, by induction, there exists $\CT\subseteq_Q\CS$ such that%
\footnote{In fact, we may construct such a~$\CT$ via a cascade of generic kernel subrepresentations.
If $\hom_{Q,\CF_\CS,\CG}(\CS,\CW)=0$ then $\ext_{Q,\CF_\CS,\CG}(\CS,\CW) = -\eul_{Q,\CF_\CS,\CG}(\CS,\CW)$, so we can choose $\CT=\CS$.
Otherwise, we continue recursively with a generic kernel subrepresentation for the pair $(\CS,\CW)$.}
\begin{align*}
  \ext_{Q,\CF_\CS,\CG}(\CS,\CW) = -\eul_{Q,\CF_\CT,\CG}(\CT,\CW).
\end{align*}
By \cref{eq:ext=ext}, it follows that
\begin{align*}
  \ext_{Q,\CF,\CG}(\CV,\CW) = -\eul_{Q,\CF_\CT,\CG}(\CT,\CW).
\end{align*}
Thus, \cref{eq:ext via eul ieq} is saturated for~$\CT$.
Since also $\CT\subseteq_Q\CV$ by \cref{lem:transitive}, this concludes the proof.
\end{proof}

\begin{proof}[Proof of \cref{claim:dim hom via eul}]
Abbreviate $\boldOmega^0=\boldOmega^0(\CS,\CF)$.
Consider the variety
\begin{align*}
  \mathbb Y_p = \{ (\Phi,v,w) \in \mathbb Y : \ker\Phi\in\boldOmega^0 \}.
\end{align*}
Here we do \emph{not} assume that~$(v,w)$ belong to~$Z$, so it does \emph{not} follow that~$\mathbb Y_p$ is contained in~$\mathbb Y_q$.
However, $\mathbb Y_p \cap \mathbb Y_q$ is a nonempty Zariski-open subset of both varieties.
Consider
\begin{align*}
  \mathbb V = \{ \Phi \in \g_{Q,\CF,\CG}(\CV,\CW) : \ker\Phi\in\boldOmega^0 \}.
\end{align*}
This is a $B_Q(\CV,\CF)$-equivariant bundle over the homogeneous space~$\boldOmega^0$.
The fibers can be identified with the injective maps in $\g_{Q,\CF_{\CV/\CS},\CG}(\CV/\CS,\CW)$ (by construction, this is a nonempty open subset).
Thus, $\mathbb V$ is a smooth irreducible variety of dimension
\begin{align}\label{eq:dim V}
  \dim\mathbb V = \dim\boldOmega^0 + \dim \g_{Q,\CF_{\CV/\CS},\CG}(\CV/\CS,\CW).
\end{align}
We claim that the projection
\begin{align*}
  p\colon \mathbb Y_p\to\mathbb V, \; (\Phi,v,w) \mapsto \Phi
\end{align*}
defines a vector bundle.
To see this, consider the fiber at some~$\Phi$ with $\ker\Phi=\CS$ (by equivariance, this is without loss of generality), which consists of the $(v,w)$ such that $\Phi v=w\Phi$.
To implement this condition, choose a complement~$\CT$ of~$\CS$ in~$\CV$ and denote $\CM=\Phi\CT$.
Then we have $v \CS \subseteq\CS$, while on $\CT$, $\Phi$ is an isomorphism onto $\CM$, so we find that $w(m)=\Phi(v(\Phi^{-1}(m)))$ for all $m\in\CM$.
If we also choose a complement~$\CN$ of~$\CM$ in~$\CW$ then we can write
\begin{align*}
  v = \begin{pmatrix}v_{00} & v_{01} \\ 0 & v_{11}\end{pmatrix}, \quad
  w = \begin{pmatrix}w_{00} & w_{01} \\ 0 & w_{11}\end{pmatrix}.
\end{align*}
with respect to $\CV=\CS\oplus\CT$ and $\CW=\CM\oplus\CN$, where $w_{00}$ is determined by $v_{00}$ (and~$\Phi$);
all other entries are completely arbitrary.
Thus, the fibers of~$p$ are vector spaces of dimension
\begin{equation}\label{eq:dim fiber over V}
\begin{aligned}
&\dim \CH_Q(\CV) - \dim \CH_Q(\CS,\CV/\CS) \\
+ &\dim \CH_Q(\CW) - \dim \CH_Q(\CV/\CS,\CW),
\end{aligned}
\end{equation}
and we obtain that $\mathbb Y_p$ is a vector bundle over the smooth irreducible variety~$\mathbb V$, hence itself smooth and irreducible.
Combining \cref{eq:dim V,eq:dim fiber over V}, we find that
\begin{align*}
  \dim \mathbb Y_p
&= \dim\boldOmega^0 + \dim \g_{Q,\CF_{\CV/\CS},\CG}(\CV/\CS,\CW) + \dim \CH_Q(\CV) \\
&- \dim \CH_Q(\CS,\CV/\CS) + \dim \CH_Q(\CW) - \dim \CH_Q(\CV/\CS,\CW).
\end{align*}
Since $\mathbb Y_p\cap\mathbb Y_q$ is a nonempty Zariski-open subset of both irreducible varieties, this is also the dimension of~$\mathbb Y_q$.
Comparing with \cref{eq:dim Y_q},
\begin{align*}
\hom_{Q,\CF,\CG}(\CV,\CW)
&= \dim\boldOmega^0 + \dim \g_{Q,\CF_{\CV/\CS},\CG}(\CV/\CS,\CW) \\
&- \dim \CH_Q(\CS,\CV/\CS) - \dim \CH_Q(\CV/\CS,\CW)
\end{align*}
and using \cref{def:filtered eul,eq:eul is edim} we obtain \cref{claim:dim hom via eul}.
\end{proof}

\begin{proof}[Proof of \cref{claim:hom ieq}]
Let $s\in \CH_Q(\CS)$ and $w\in \CH_Q(\CW)$ such that
\begin{align*}
  \dim \Hom_{Q,\CF_\CS,\CG}(s,w)
= \hom_{Q,\CF_\CS,\CG}(\CS,\CW).
\end{align*}
Here, $w$ can vary in an open subset of~$\CH_Q(\CW)$.
Thus, by definition of the generic kernel subrepresentation~$\CS$, there exists~$v\in\CH_Q(\CV)$ and~$\Phi\in\Hom_{Q,\CF,\CG}(v,w)$ such that~$(v,w)\in Z$ and~$\ker\Phi\in\boldOmega^0(\CS,\CF)$.
By $B_Q(\CV,\CF)$-equivariance, we may assume that~$\ker\Phi=\CS$.

Since $\CS$ is a subrepresentation of~$v$, we can consider the quotient maps~$\bar v\colon\CV/\CS\to\CV/\CS$ and $\bar\Phi\in\Hom_{Q,\CF_{\CV/\CS},\CG}(\bar v,w)$.
The latter is injective, so composition with $\bar\Phi$ defines an injective map
\begin{align*}
  \Hom_{Q,\CF_\CS,\CF_{\CV/\CS}}(s,\bar v)
\hookrightarrow \Hom_{Q,\CF_\CS,\CG}(s,w).
\end{align*}
Thus:
\begin{align*}
  \dim \Hom_{Q,\CF_\CS,\CG}(s,w)
&\geq \dim \Hom_{Q,\CF_\CS,\CF_{\CV/\CS}}(s,\bar v) \\
&\geq \hom_{Q,\CF_\CS,\CF_{\CV/\CS}}(\CS,\CV/\CS),
\end{align*}
which concludes the proof.
\end{proof}

\begin{example}\label{ex:quiver W2 red}
Consider the quiver~$W_2$ from \cref{ex:quiver W2}.
Let $\CV=(\C,\C)$ and choose $\CF$ to be the standard filtration.
Then we can identify $\CH_Q(\CV)=\C^2$ and $\g_{Q,\CF,\CF}=\C^2$.
Given $(v_1,v_2), (w_1,w_2)\in\CH_Q(\CV)$ and $(\Phi_1,\Phi_2)\in\g_{Q,\CF,\CF}$, the condition that $\Phi\in\Hom_{Q,\CF,\CF}(v,w)$ means that
\begin{align*}
  \Phi_2 v_1 = w_1 \Phi_1
  \quad\text{and}\quad
  \Phi_2 v_2 = w_2 \Phi_1.
\end{align*}
Thus, the variety~$\mathbb Y$ in the proof of \cref{thm:ext via eul} is
\begin{align*}
  \mathbb Y
= \{ \Phi_1 = \Phi_2 = 0 \} \;\cup\; \{ v_1 w_2 - v_2 w_1 = 0, \, \Phi_2 v_1 = w_1 \Phi_1 \}
\end{align*}
so $\mathbb Y$ has two irreducible components, each of dimension~4.
\end{example}

\begin{remark}\label{rem:opt test}
In the minimization of \cref{thm:ext via eul}, we only need to consider families of subspaces~$\CS$ that can arise as generic kernel subrepresentations, as well as possibly~$\CS=(\{0\})$ and~$\CS=\CV$.
In many examples, this allows to a priori restrict the minimization to families with particular properties.

For example, suppose that $\dim V_x = \dim V_y$ and $\dim W_x = \dim W_y$ for one or more arrows $a\colon x\to y\in A$.
Then, for generic $v\in\CH_Q(\CV)$ and~$w\in\CH_Q(\CW)$, the corresponding components~$v_a$ and~$w_a$ are isomorphisms, so~$\Phi_y = w_a \Phi_x v_a^{-1}$ and $\dim\ker\Phi_x = \dim \ker\Phi_y$.
Thus, in this case we can restrict the minimization to subspaces~$\CS$ that satisfy $\dim S_x=\dim S_y$ for each such arrow.
\end{remark}

\section{Proof of the main theorem}\label{sec:6}
In this section, we will establish \cref{thm:main}.
In fact, we will prove a refined version, which asserts that we only need to consider subspaces for which the generic intersection variety consists of a single point:

\begin{theorem}\label{thm:main refined}
Let $(\CV,\CF)$ be a filtered dimension vector and $\CS$ a family of subspaces as above.
Then, $\CS \subseteq_Q \CV$ if and only if
\begin{enumerate}[label=\emph{(\Alph*)},ref={(\Alph*)}]
\item\label{it:main refined A} $\edim_{Q,\CF}(\CS,\CV) \geq 0$,
\item\label{it:main refined B} $\CT \subset_Q \CV$ for every $\CT\in P_Q(\CS,\CF_\CS)$, $\CT\neq\CS$.
\end{enumerate}
\end{theorem}

We will need some intermediate results to prove \cref{thm:main refined}.
To test if some $\CS$ is $Q$-intersecting, we need to in principle consider generic representations in~$\CH_Q(\CV)$.
We first show that there exists a \emph{universal} representation that tests $Q$-intersection.

\begin{lemma}\label{lem:detecting}
There exists a nonempty Zariski-open set of $v^*\in\CH_Q(\CV)$ with the following property:
For every $\CS\subseteq\CV$, we have that $\CS\subseteq_Q\CV$ if and only if there exists $\CT\in\boldOmega^0(\CS,\CF)$ such that $v^*\CT\subseteq\CT$.
\end{lemma}

\noindent
We say that $v^*$ is \emph{detecting $Q$-intersection} in~$\CV$.

\begin{proof}
Consider the finitely many Schubert cells of the Grassmannians~$\Gr_Q(\boldalpha,\CV)$, where~$\boldalpha$ ranges over all dimension vectors~$\boldalpha\leq\dim\CV$.
For each Schubert cell $\boldOmega^0$, denote by $\boldOmega$ its closure and define
\begin{align*}
  \CH_Q^{\boldOmega^0}
&= \{ v \in \CH_Q(\CV) : \exists \CT\in\boldOmega^0 \text{ such that } v\CT\subseteq\CT \}.
\end{align*}
By \cref{lem:generic}, if $\boldOmega$ is $Q$-intersecting then $\CH_Q^{\boldOmega^0}$ contains a nonempty Zariski-open set, while it is otherwise not Zariski-dense.
Thus,
\begin{align*}
  \widetilde\CH_Q
\coloneqq \bigcap_{\boldOmega \not\subseteq_Q \CV} \overline{\CH_Q^{\boldOmega^0}}^c
\cap \bigcap_{\boldOmega \subseteq_Q \CV} \CH_Q^{\boldOmega^0}
\end{align*}
contains a nonempty Zariski-open set.
By construction, every $v^*\in \widetilde\CH_Q$ is detecting $Q$-intersection in~$\CV$.
\end{proof}

Next, we show that we can by an optimization procedure construct Schubert cells for which the generic intersection variety consists of a single point only.
Recall that $d(\CN) = \sum_{x\in Q_0} \dim N_x$ denotes the total dimension of a family of vector spaces.

\begin{definition}[Slope]
Let $(\CV,\CF)$ and $(\CW,\CG)$ be filtered dimension vectors.
We define the \emph{slope} of a nonzero subquotient~$\CN$ of~$\CV$ by
\begin{align*}
  \sigma(\CN) \coloneqq \frac1{d(\CN)} \eul_{Q,\CF_\CN,\CG}(\CN,\CW),
\end{align*}
where $\CF_\CN$ denotes the filtration induced by~$\CF$ on~$\CN$.
\end{definition}

\noindent For fixed $v\in\CH_Q(\CV)$, consider the set of subrepresentations of arbitrary dimension,
\begin{align*}
 \BS(v) \coloneqq \{ \CS\subseteq\CV : v\CS\subseteq\CS \}.
\end{align*}
Note that $\BS(v)$ is closed under vector space sum and intersection.

\begin{proposition}\label{prp:unique maximin}
Let $(\CV,\CF)$ and $(\CW,\CG)$ be filtered dimension vectors and let $v^*\in\CH_Q(\CV)$ be an element detecting $Q$-intersection in~$\CV$.
Define $\sigma^*=\min_{(\{0\})\neq\CS\in\BS(v^*)} \sigma(\CS)$ and $d^*=\max_{(\{0\})\neq\CS\in\BS(v^*), \sigma(\CS)=\sigma^*} d(\CS)$.
Then there exists a unique family $\CS^*\in\BS(v^*)$ such that $\sigma(\CS)=\sigma^*$ and $d(\CS)=d^*$.
\end{proposition}
\noindent
We call $\CS^*$ the \emph{maximin subrepresentation} for~$v^*$;
it is $Q$-intersecting in~$\CV$.
\begin{proof}
Existence is clear, so we only argue for uniqueness.
Suppose for sake of finding a contradiction that $\CS_1$ and $\CS_2$ are two distinct families of subspaces with the desired maximin property.
Consider the short exact sequence
\begin{align*}
  0 \to \CS_1\cap\CS_2 \to \CS_1 \to \CS_1 / (\CS_1\cap\CS_2) \to 0.
\end{align*}
If $\CS_1\cap \CS_2\neq(\{0\})$ then
\begin{align*}
  \sigma(\CS_1) = \frac{d(\CS_1\cap \CS_2)}{d(\CS_1)} \sigma(\CS_1\cap \CS_2) + \frac{d(\CS_1 / (\CS_1\cap\CS_2))}{d(\CS_1)} \sigma(\CS_1 / (\CS_1\cap\CS_2))
\end{align*}
as follows from \cref{eq:add eul}.
Thus, $\sigma(\CS_1)$ is a convex combination of slopes.
By minimality, $\sigma(\CS_1) \leq \sigma(\CS_1\cap \CS_2)$, hence we find that
\begin{align}\label{eq:first slope bound}
  \sigma(\CS_1) \geq \sigma(\CS_1 / (\CS_1\cap\CS_2)).
\end{align}
This inequality also holds when $\CS_1\cap\CS_2=(\{0\})$.
Next, consider
\begin{align*}
  0 \to \CS_2 \to \CS_1+\CS_2 \to (\CS_1+\CS_2)/\CS_2 \to 0.
\end{align*}
Since $\CS_1\neq\CS_2$, $d(\CS_1+\CS_2)>d(\CS_2)$, so $\sigma(\CS_1+\CS_2)>\sigma(\CS_2)$ by extremality.
Thus, by the same argument,
\begin{align*}
  \sigma((\CS_1+\CS_2)/\CS_2)
> \sigma(\CS_2) = \sigma(\CS_1).
\end{align*}
Together with \cref{eq:first slope bound}, we obtain
\begin{align*}
  \sigma((\CS_1+\CS_2)/\CS_2) > \sigma(\CS_1 / (\CS_1\cap\CS_2)).
\end{align*}
As vector spaces, both quotients are isomorphic and hence have the same dimension vector and total dimension.
Thus, it follows from the definition of the slope and filtered Euler number that
\begin{align*}
  \dim \g_{Q,\CF_1,\CG}((\CS_1+\CS_2)/\CS_2,\CW) > \dim \g_{Q,\CF_2,\CG}(\CS_1 / (\CS_1\cap\CS_2),\CW),
\end{align*}
where we abbreviate the induced filtrations by $\CF_1$ and $\CF_2$.
However, the natural isomorphism that interprets each~$\bar\Phi\colon (\CS_1+\CS_2)/\CS_2\to\CW$ as a map $\CS_1/(\CS_1\cap\CS_2)\to\CW$ restricts to an injection
\begin{align*}
  \g_{Q,\CF_1,\CG}((\CS_1+\CS_2)/\CS_2,\CW) \hookrightarrow \g_{Q,\CF_2,\CG}(\CS_1 / (\CS_1\cap\CS_2),\CW),
\end{align*}
since if $\Phi\colon \CS_1+\CS_2\to\CW$ is a representative of some~$\bar\Phi$ then $\Phi((\CS_1+\CS_2) \cap \CF(i)) \subseteq \CG(i)$ implies that $\Phi(\CS_1 \cap \CF(i)) \subseteq \CG(i)$ for all~$i$.
This is the desired contradiction.
\end{proof}

\begin{lemma}\label{lem:maximin in same cell}
In the situation of \cref{prp:unique maximin}, the slope~$\sigma^*$ and dimension~$d^*$ of the maximin subrepresentation do not depend on the choice of~$v^*$.
Moreover, the maximin subrepresentations obtained by varying~$v^*$ are all in the same Schubert cell.
\end{lemma}
\begin{proof}
Consider another $v^\#\in\CH_Q(\CV)$ that detects $Q$-intersection and let~$\CS^\#$ denote the corresponding maximin subrepresentation.
Since $\CS^*$ is $Q$-intersecting, there exists some~$\CT\in\boldOmega^0(\CS^*,\CF)$ such that~$v^\#\CT \subseteq \CT$.
Then $\sigma(\CT)=\sigma(\CS^*)$, since the Euler number only depends on the Schubert cell, and hence~$\sigma(\CS^*) \geq \sigma(\CS^\#)$.
Running the argument in reverse, we obtain that $\sigma(\CS^*) = \sigma(\CS^\#)$.
We similarly find that~$d(\CS^*) = d(\CS^\#)$, so $\CS^\#=\CT\in\boldOmega^0(\CS^*,\CF)$, which confirms the last statement.
\end{proof}

\begin{proposition}\label{prp:maximin point}
In the situation of \cref{prp:unique maximin}, the maximin subrepresentation~$\CS^*$ is in~$P_Q(\CV,\CF)$.
\end{proposition}
\begin{proof}
We abbreviate $\boldOmega=\boldOmega(\CS^*,\CF)$.
It suffices to argue that $\boldOmega_{v^\#}$ is a single point for every~$v^\#\in\CH_Q(\CV)$ that is detecting $Q$-intersection (a nonempty Zariski-open set according to \cref{lem:detecting}).
We will show that $\boldOmega_{v^\#} = \{\CS^\#\}$, where $\CS^\#$~denotes the maximin subrepresentation.

Indeed, $\CS^\#$ is a subrepresentation of~$v^\#$ and, by \cref{lem:maximin in same cell}, belongs to the same Schubert cell as~$\CS^*$, so $\CS^\#\in\boldOmega_{v^\#}$.
Conversely, suppose that~$\CT\in\boldOmega_{v^\#}$.
Since it is in the same Grassmannian as $\CS^\#$, we have that~$d(\CT)=d(\CS^\#)$ and $\dim\CH_Q(\CT,\CW)=\dim\CH_Q(\CS^\#,\CW)$.
Moreover,
\begin{align}\label{eq:dim ieq sharp}
  \dim \g_{Q,\CF_\CT,\CG}(\CT,\CW) \leq \dim \g_{Q,\CF_\CS^\#,\CG}(\CS^\#,\CW).
\end{align}
Indeed, since $\CT$ is in the closure of the $B_Q(\CV,\CF)$-orbit of~$\CS^\#$, it is clear that, for each $x\in Q_0$ and $i$, $\dim T_x \cap F_x(i) \geq \dim S^\#_x \cap F_x(i)$, so \cref{eq:dim ieq sharp} follows from \cref{lem:dim filtered morphisms}.
Thus, $\sigma(\CT) \leq \sigma(\CS^\#)$ and $d(\CT)=d(\CS^\#)$.
As a consequence, $\CT=\CS^\#$ by the uniqueness of the maximin subrepresentation.
We conclude that $\boldOmega_{v^\#} = \{\CS^\#\}$, as we set out to prove.
\end{proof}

We thus obtain the following result, which strengthens the main conclusion of \cref{thm:ext via eul}.

\begin{corollary}\label{cor:diabolic}
Let $(\CV,\CF)$ and $(\CW,\CG)$ be filtered dimension vectors such that~$\ext_{Q,\CF,\CG}(\CV,\CW)>0$.
Then there exists a family~$\CT^*\in P_Q(\CV,\CF)$ such that $\eul_{Q,\CF_{\CT^*},\CG}(\CT^*,\CW)<0$.
\end{corollary}
\begin{proof}
Let $v^*\in\CH_Q(\CV)$ be an element detecting $Q$-intersection.
By \cref{thm:ext via eul}, there exists $\CT\subseteq_Q\CV$ such that $\eul_{Q,\CF_\CT,\CG}(\CT,\CW)<0$.
Thus,~$\BS(v^*)$ contains an element of negative slope.
As a consequence, the maximin subrepresentation~$\CT^*$ also has negative slope, hence negative Euler number.
By \cref{prp:maximin point}, it belongs to~$P_Q(\CV,\CF)$.
\end{proof}

We now prove the main result of this article.

\begin{proof}[Proof of \cref{thm:main refined}]
As discussed before, \cref{lem:necessary,lem:transitive} show that if $\CS$ is $Q$-intersecting in~$\CV$ then~\ref{it:main refined A} and~\ref{it:main refined B} are necessarily satisfied.

We now prove the converse.
Suppose that $\CS$ is \emph{not} $Q$-intersecting in~$\CV$.
By \cref{thm:Q-intersecting iff ext=0}, this means that $\ext_{Q,\CF_\CS,\CF_{\CV/\CS}}(\CS,\CV/\CS)>0$.
Therefore, \cref{cor:diabolic} shows that there exists~$\CT\in P_Q(\CS,\CF_\CS)$ such that
\begin{align}\label{eq:eul < 0}
  \eul_{Q,\CF_\CT,\CF_{\CV/\CS}}(\CT,\CV/\CS)<0.
\end{align}
If $\CT=\CS$, this filtered Euler number equals $\edim_{Q,\CF}(\CS,\CV)$ (\cref{eq:eul is edim}), so~\ref{it:main refined A} is violated.
We will therefore assume that $\CT\subset\CS$.
In this case,
\begin{align*}
  \edim_{Q,\CF}(\CT,\CV)
&= \eul_{Q,\CF_\CT,\CF_{\CV/\CT}}(\CT,\CV/\CT) \\
&= \eul_{Q,\CF_\CT,\CF_{\CV/\CT}}(\CT,\CV/\CT) - \edim_{Q,\CF_\CS}(\CT,\CS) \\
&= \eul_{Q,\CF_\CT,\CF_{\CV/\CT}}(\CT,\CV/\CT) - \eul_{Q,\CF_\CT,\CF_{\CS/\CT}}(\CT,\CS/\CT) \\
&= \eul_{Q,\CF_\CT,\CF_{\CV/\CS}}(\CT,\CV/\CS) < 0,
\end{align*}
where the first equality is \cref{eq:eul is edim},
the second equality holds because~$\CT\in P_Q(\CS,\CF_\CS)$ and so $\edim_{Q,\CF_\CS}(\CT,\CS)=0$ (see discussion below \cref{def:point}),
the next steps are \cref{eq:eul is edim} and \cref{eq:add eul 2nd},
and we finally used \cref{eq:eul < 0}.
Thus, $\edim_{Q,\CF}(\CT,\CV)<0$, which by \cref{lem:necessary} implies that $\CT$ is not intersecting in $\CV$.
This shows that~\ref{it:main refined B} is violated.
\end{proof}

\begin{remark}\label{rem:opt test 2}
One can in specific cases further constrain the families~$\CT$ that need to be considered in condition~\ref{it:main refined B} of \cref{thm:main refined} by careful inspection of the maximin construction and using \cref{rem:opt test}.
For example, we may always restrict to~$\CT$ that satisfy $\dim T_x=\dim T_y$ for every arrow~$a\colon x\to y\in Q_1$ such that $\dim S_x=\dim S_y$ and $\dim V_x=\dim V_y$.

To see this, recall that the subspaces~$\CT$ were produced by applying \cref{cor:diabolic} to $\CS$ and $\CV/\CS$.
In the proof of \cref{cor:diabolic}, we first invoked \cref{thm:ext via eul} to obtain an element $\CT\subseteq_Q\CS$ with $\eul_{Q,\CF_\CT,\CF_{\CV/\CS}}(\CT,\CV/\CS)<0$ and then used the maximin construction of \cref{prp:unique maximin} to find an element in~$P_Q(\CS,\CF_\CS)$ with negative Euler number.
Since $\dim V_x/S_x = \dim V_y/V_y$, we may by \cref{rem:opt test} assume that $\dim T_x = \dim T_y$ for each arrow as above.
We would like to restrict the maximin construction to the subset~$\BS'\subseteq\BS(v^*)$ consisting of families that satisfy this dimension condition.
For \emph{generic}~$v^*$ that detect $Q$-intersection in~$\CS$, $\BS'$ is closed under vector space sum and intersection, as follows by a similar argument as given in \cref{rem:opt test}.
Thus, the same proofs as given above allow us to conclude that there exists a unique maximin subrepresentation~$\CT^*$ (with possibly different $\sigma^*<0$ and $d^*>0$) which is an element of~$P_Q(\CS,\CF_\CS)$ and moreover satisfies~$\dim T^*_x = \dim T^*_y$ for each arrow as above.

In the case of the Horn quiver, this optimization recovers Belkale's conditions for intersections of Schubert classes of the Grassmannian (\cref{sec:horn}).
\end{remark}

By the same reasoning, but working with families of subspaces without filtrations, one can prove a refined version of Schofield's theorem~\cite{MR1162487}.
To state the result, write $\boldalpha\leq_Q\boldn$ if $\boldalpha$ is a Schofield subdimension vector of~$\boldn$, and define $P_Q(\boldalpha)$ as the set of subdimension vectors $\boldbeta\leq\boldalpha$ such that $\Gr_Q(\boldbeta,\boldalpha)_v$ is a point for generic $v\in\CH_Q(\boldalpha)$.

\begin{theorem}\label{thm:schofield refined}
Let $\boldalpha$ be a subdimension vector of some dimension vector $\boldn$.
Then, $\boldalpha\leq_Q\boldn$ if and only if
\begin{enumerate}[label=\emph{(\Alph*)},ref={(\Alph*)}]
\item $\braket{\boldalpha,\boldn-\boldalpha} \geq 0$,
\item $\boldbeta\le_Q\boldn$ for every $\boldbeta\in P_Q(\boldalpha)$, $\boldbeta\neq\boldalpha$.
\end{enumerate}
\end{theorem}

\section{Horn conditions for Q-intersection}\label{sec:horn}
\Cref{thm:main refined} can readily be translated into a recursive algorithm for deciding $Q$-intersection that only involves the easily computable expected dimensions (\cref{def:edim}).

\begin{definition}[Horn set]\label{def:horn set}
Let $(\CV,\CF)$ be a filtered dimension vector.
We define $\Horn_Q(\CV,\CF)$ inductively as the set of~$\CS\subseteq\CV$ such that, if~$\CS\neq\CV$,
\begin{enumerate}[label=\emph{(\Alph*)},ref={(\Alph*)}]
\item $\edim_{Q,\CF}(\CS,\CV) \geq 0$,
\item\label{it:horn set B} $\edim_{Q,\CF}(\CT,\CV) \geq 0$ for every $\CT\in\Horn_Q(\CS,\CF_\CS)$ such that $\CT\neq\CS$ and~$\edim_{Q,\CF_\CS}(\CT,\CS)=0$.
\end{enumerate}
\end{definition}

\begin{theorem}[Horn conditions]\label{thm:horn}
Let $(\CV,\CF)$ be a filtered dimension vector and $\CS\subseteq\CV$ a family of subspaces.
Then, $\CS\subseteq_Q\CV$ if and only if $\CS\in\Horn_Q(\CV,\CF)$.
\end{theorem}
\begin{proof}
This follows by induction over the total dimension of~$\CS$.
Indeed, suppose that we have proved the result for any~$\CT\subset\CS$.
Then the `if` follows from \cref{thm:main refined}, while the `only if' is a consequence of \cref{lem:necessary,lem:transitive}.
\end{proof}


It is clear that in condition~\ref{it:horn set B} of \cref{def:horn set} we only need to consider subspaces~$\CT$ that belong to $P_Q(\CS,\CF)$.
However, it is much harder to check membership in~$P_Q(\CS,\CF_\CS)$ (i.e., whether the generic intersection variety is a point) than to compute the expected dimension and check that $\edim_{Q,\CF_\CS}(\CT,\CS)=0$ (i.e., whether the generic intersection variety is a finite set of points).

\subsection{Combinatorial Horn conditions}\label{subsec:combi horn}
Since the property of being $Q$-in\-ter\-sec\-ting only depends on the Schubert cell, we can also give a combinatorial version of the above characterization.
We will work in the following setup:
For every finite subset~$J = \{i_1 < \dots < i_\ell\} \subseteq\N$, define the vector space~$V(J) = \bigoplus_{j\in J} \C e_j$ and the filtration~$F(J)$ with elements~$F(J)(a) = \bigoplus_{k=1}^a \C e_{j_k}$ for $a=1,\dots,\ell$.
Thus, every collection~$\CJ=(J_x)_{x\in Q_0}$ of finite subsets $J_x\subseteq\N$ defines a family of vector spaces $\CV(\CJ)$ and a family of filtrations $\CF(\CJ)$, i.e., a filtered dimension vector.

We will write $\CK\subseteq\CJ$ if $\CK=(K_x)_{x\in Q_0}$ is a family of subsets $K_x\subseteq J_x$ for every~$x\in Q_0$.
In this case, $\CV(\CK)$ is a family of subspaces of~$\CV(\CJ)$.
As discussed on \cpageref{eq:concrete schubert}, every Schubert cell in a Grassmannian of~$\CV(\CJ)$ is the Borel orbit of some family of the form~$\CV(\CK)$.
Let us also write~$\boldOmega(\CK)$ for the corresponding Schubert variety defined by $\CV(\CK)$.

We write~$\CK\subseteq_Q\CJ$ if $\CV(\CK)$ is \emph{$Q$-intersecting} in~$\CV(\CJ)$, and we abbreviate the \emph{expected dimension} by $\edim_Q(\CK,\CJ) = \edim_{Q,\CF(\CJ)}(\CV(\CK),\CV(\CJ))$.
Using \cref{eq:schubert dim}, the expected dimension can be computed as follows:
\begin{equation}\label{eq:edim concrete}
\begin{aligned}
  \edim_Q(\CK,\CJ)
&= \sum_{x\in Q_0} \sum_{j\in K_x} \bigl( p_j(J_x) - p_j(K_x) \bigr) \\
&- \!\!\!\!\!\sum_{a:x\to y\in Q_1}\!\!\!\!\! |K_x| \bigl( |J_y| - |K_y| \bigr),
\end{aligned}
\end{equation}
where we write $p_x(S)$ for the position of an element~$x$ in a set~$S$ in increasing order, i.e., $p_x(S)=1$ for the smallest element $x\in S$, etc.
We obtain a simple practical criterion for deciding $Q$-in\-ter\-sec\-tion:

\begin{definition}[Combinatorial Horn set]\label{def:combinatorial horn set}
Let $\CJ=(J_x)_{x\in Q_0}$ be a family of finite subsets of $\N$.
We define $\Horn_Q(\CJ)$ as the set of $\CK\subseteq\CJ$ such that, if~$\CK\neq\CJ$,
\begin{enumerate}[label=\emph{(\Alph*)},ref={(\Alph*)}]
\item $\edim_Q(\CK,\CJ)\geq0$,
\item $\edim_Q(\CL,\CJ)\geq0$ for every $\CL\in\Horn_Q(\CK)$ that satisfies $\CL\neq\CK$ and $\edim_Q(\CL,\CK)=0$.
\end{enumerate}
\end{definition}

\begin{theorem}[Combinatorial Horn conditions]\label{thm:horn concrete}
Let $\CJ=(J_x)_{x\in Q_0}$ be a family of finite subsets of~$\N$ and $\CK\subseteq\CJ$ a family of subsets.
Then, $\CK\subseteq_Q\CJ$ if and only if~$\CK\in\Horn_Q(\CJ)$.
Moreover, if $\CK\subseteq_Q\CJ$ then the generic intersection variety is of dimension $\edim_Q(\CK,\CJ)$.
\end{theorem}

It is straightforward to incorporate the optimizations discussed in \cref{rem:opt test,rem:opt test 2} into this criterion.
Given a family~$\CJ$ that satisfies~$|J_x|=|J_y|$ for every arrow $x\to y$ in some subset $A\subseteq Q_1$, define $\Horn_{Q,A}(\CJ)$ inductively as the set of $\CK\subseteq\CJ$ satisfying the same dimension condition (i.e., $|K_x| = |K_y|$ for every arrow $x\to y\in A$) and, if $\CK\neq\CJ$,
\begin{enumerate}[label=(\Alph*)]
\item $\edim_Q(\CK,\CJ)\geq0$,
\item $\edim_Q(\CL,\CJ)\geq0$ for every $\CL\in\Horn_{Q,A}(\CK)$ with $\CL\neq\CK$ and $\edim_Q(\CL,\CK)=0$.
\end{enumerate}
Then, the elements of $\Horn_{Q,A}(\CJ)$ are precisely the $Q$-intersecting subfamilies of~$\CJ$ that satisfy the dimension condition.

We now specialize our result to the Horn quiver~$H_s$ from~\eqref{eq:horn quiver} and constant dimension vectors (corresponding to the choice where~$A$ contains all arrows of $H_s$).
Thus, let~$\CJ$ denote a family of~$s+1$ subsets of $\N$, each of cardinality~$n$, and $\CK\subseteq\CJ$ a collection of subsets, each of cardinality~$0\leq r\leq n$.
If we identify each $V(J_x) \cong \C^n$, each $V(K_x)$ determines a Schubert variety~$\Omega(K_x)$ in~$\Gr(r,n)$.
As explained in \cref{ex:horn quiver}, the Schubert classes~$[\Omega(K_x)]_{x=1,\dots,s+1}$ are intersecting if and only if~$\CK\subseteq_{H_s}\CJ$.
Thus, we obtain the following necessary and sufficient condition for Schubert varieties in $\Gr(r,n)$ to intersect:

\begin{definition}[Belkale's Horn set]
Let $\CJ$ denote a family of~$s+1$ subsets of $\N$, each of cardinality~$n$, and $1\leq r\leq n$.
We define $\Belkale_s(r,\CJ)$ as the set of $\CK\subseteq\CJ$ such that each $K_x$ has cardinality~$r$ and,
\begin{enumerate}[label=\emph{(\Alph*)},ref={(\Alph*)}]
\item $\edim_{H_s}(\CK,\CJ)\geq0$,
\item $\edim_{H_s}(\CL,\CJ)\geq0$ for every $\CL\in\Belkale_s(d,\CK)$ where $1\leq d<r$ and $\edim_{H_s}(\CL,\CK)=0$.
\end{enumerate}
\end{definition}

Note that for the quiver~$H_s$, $\CJ=(J_x)$ with $J_x=\{1,\dots,n\}$ for all~$x$, and $\CK\subseteq\CJ$ such that each $K_x$ has cardinality~$r$, \cref{eq:edim concrete} simplifies to
\begin{align*}
  \edim_{H_s}(\CK,\CJ) = \left( \sum_{x=1}^{s+1} \sum_{a=1}^r \bigl( K_x(a) - a \bigr) \right) - s r (n-r),
\end{align*}
where $K_x(1) < \dots < K_x(r)$ denote the elements of $K_x$.
This coincides with Belkale's definition of the expected dimension~\cite{MR2177198}.

\begin{theorem}[Belkale]\label{thm:belkale}
Let $1\leq r\leq n$, $\CJ$ a family of~$s+1$ subsets of $\N$, each of cardinality~$n$, and $\CK\subseteq\CJ$ a family of subsets, each of cardinality~$r$.
Then, $\CK\subseteq_{H_s}\CJ$ if and only if $\CK\in\Belkale_s(r,\CJ)$.
\end{theorem}

In his original proof~\cite{MR2177198}, Belkale constructs an element~$\CT\subset_Q\CV$ with constant $\dim\CT$, by a `cascade construction' of generic kernels (a priori different from the one we used) such that $\CT$ fails to satisfy the Horn conditions if the Schubert classes are not intersecting.
Belkale's proof has been simplified by Sherman~\cite{sherman1}, as explained in~\cite{BVW}.

\subsection{Relation to augmented quivers}\label{subsec:augmented}
We now discuss the relation between our criterion and the construction of Derksen-Weyman in more detail (cf.~\cref{subsec:schofield}).

Consider a quiver~$Q$ and a dimension vector~$\boldn$, and define $\CJ=(J_x)_{x\in Q_0}$ by $J_x=\{1,\dots,n_x\}$.
Inspired by Derksen-Weyman~\cite{MR1758750}, define an \emph{augmented quiver}~$\tilde Q$ in the following way.
For each vertex $x\in Q_0$, introduce additional vertices~$(x,i)$ for $i=1,\dots,n_x-1$, and add arrows
\begin{align*}
  (x,1)\longrightarrow \ldots\longrightarrow (x,n_x-1) \longrightarrow (x,n_x) = x.
\end{align*}
Define the dimension vector~$\boldtilden$ with components $\tilde n_{x,i} = i$.
Note that~$\boldtilden$ coincides with $\boldn$ on~$Q$.
Given a family of subsets~$\CK\subseteq\CJ$, we can similarly associate a subdimension vector~$\boldtildealpha$ by
$\boldtildealpha_{x,i} = \lvert K_x \cap \{1,\dots,i\} \rvert$.

Then the correspondence between our picture and the augmented quiver picture is as follows:
$\CK\subseteq_Q\CJ$ if and only if $\boldtildealpha \leq_{\tilde Q} \boldtilden$, that is, if and only if $\boldtildealpha$ is a Schofield subdimension vector of $\boldtilden$.

Thus, one can also determine if $\CK\subseteq_Q\CJ$ by using Schofield's inductive criterion for subdimension vectors of the augmented quiver~$\tilde Q$.
This is not obviously equivalent to our \cref{thm:main,thm:main refined}, which apply to~$Q$ directly.
Indeed, even using our refinement of Schofield's criterion (\cref{thm:schofield refined}), one would a priori need to test Schofield subdimension vectors in~$P_{\tilde Q}(\boldtildealpha)$, which in general is a much larger set than $P_Q(\CV(\CK),\CF(\CK))$.

As an easy example, consider the quiver~$Q$ with a single arrow, $a\to b$.
For $\CK=(\{1,2\},\{1,2\})$, the set~$P_Q(\CV(\CK),\CF(\CK))$ has 7 elements, namely the following subfamilies of~$\CK$:
\begin{align*}
(\emptyset, \emptyset), \;
(\emptyset, 1), \;
(\emptyset, 12), \;
(1, 2), \;
(1, 12), \;
(2, 1), \;
(12, 12),
\end{align*}
where we write $12$ instead of $\{1,2\}$ etc.~to improve readability.
In contrast, for the extended quiver $(a,1) \to (a,2)\to(b,2)\leftarrow(b,1)$ and the dimension vector~$\boldtildealpha=(1,2,2,1)$ corresponding to~$\CK$, there are 12 Schofield subdimension vectors in~$P_{\tilde Q}(\boldtildealpha)$:
\begin{align*}
&(0,0,0,0),&
&(0,0,1,1),&
&(0,0,2,0),&
&(0,0,2,1),& \\
&(0,1,1,1),&
&(0,2,2,0),&
&(0,2,2,1),&
&(1,1,1,0),& \\
&(1,1,2,0),&
&(1,1,2,1),&
&(1,2,2,0),&
&(1,2,2,1).&
\end{align*}
Indeed, while every $\CL\in P_Q(\CV(\CK),\CF(\CK))$ produces an element $\boldtildebeta \in P_{\tilde Q}(\boldtildealpha)$ by $\tilde\beta_{x,i} = |L_x \cap \{1,\dots,i\}|$, it is clear that only elements with
\begin{align}\label{eq:beta jump}
  \tilde\beta_{x,i} \leq \tilde\beta_{x,i+1} \leq \tilde\beta_{x,i} + 1
\end{align}
can arise in this way.

While \cref{thm:main refined} is not a consequence of Schofield's theorem, it is possible to give an alternative proof using the augmented quiver construction, staying purely in the realm of ordinary dimension vectors.
Indeed, using similar arguments as in \cref{rem:opt test,rem:opt test 2} one can prove a refined version of Schofield's theorem (or \cref{thm:schofield refined}) for dimension vectors of the form~$\boldtildealpha$ and~$\boldtilden$, stating that in order to determine whether $\boldtildealpha \leq_{\tilde Q} \boldtilden$, it suffices to consider~$\boldtildebeta \in P_{\tilde Q}(\boldtildealpha)$ that satisfy \cref{eq:beta jump} and hence arise from some family~$\CL\in P_Q(\CV(\CK),\CF(\CK))$.

\section{Applications to Representation Theory}\label{sec:rep theo}
In this section, we recall that the $Q$-intersecting Schubert varieties determine a complete set of inequalities characterizing the cone~$C_Q(\CV)$ generated by the highest weights of irreducible $\GL_Q(\CV)$-representations that appear in the space of polynomial functions on~$\CH_Q(\CV)$, as mentioned previously in \cref{subsec:intro rep theo}.
Applying an argument of Ressayre, we also show that the semigroup of highest weights is saturated.
Together, we obtain \cref{thm:moment cone and rep theory summary} as announced in the introduction.

We largely follow the notation of \cref{subsec:combi horn}.
Consider a quiver~$Q$ and a dimension vector~$\boldn$, and define $\CJ=(J_x)_{x\in Q_0}$ by $J_x=\{1,\dots,n_x\}$.
Let~$\CV=\CV(\CJ)$.
It is easy to see that, if the quiver~$Q$ has no cycles, then the action of $\GL_Q(\CV)$ on the space $\Sym^*(\CH_Q(\CV)$ of polynomial functions on~$\CH_Q(\CV)$ decomposes with finite multiplicities.
A basis for the Cartan subalgebra of~$\gl(V_x)$ is given by the diagonal matrices~$h_{x,i}$ for~$i=1,\dots,n_x$ such that~$h_{x,i} e_j = \delta_{i,j} e_j$ for $j=1,\dots,n_x$.
Consider~$z_x = \sum_{i=1}^{n_x} h_{x,i}$.
Then, $z=(z_x)_{x\in Q_0}$ is in~$\z=\bigoplus_{x\in Q_0} \R z_x$, the center of~$\gl_Q(\CV)$, and acts by zero in the infinitesimal action of~$\gl_Q(\CV)$ on $\CH_Q(\CK)$.
We label the dominant weights for~$\GL_Q(\CV)$ by a collection~$\boldlambda=(\lambda_x)_{x\in Q_0}$, where each~$\lambda_x$ is a function~$\{1,\ldots,n_x\}\to\ZZ$ such that~$\lambda_x(i)\geq \lambda_x(j)$ for all $1\leq i\leq j\leq n_x$.
Let~$V_{\boldlambda}$ denote the irreducible representation of $\GL_Q(\CV)$ with highest weight~$\boldlambda$.
We decompose:
\begin{align*}
  \Sym^*(\CH_Q(\CV)) = \bigoplus_{\boldlambda} m(\boldlambda) V_{\boldlambda}.
\end{align*}
Note that $V_{\boldlambda}$ occurs with nonzero multiplicity (i.e., $m(\boldlambda)>0$) if and only if there exists a nonzero homogeneous polynomial~$P$ on $\CH_Q(\CV)$ which is semi-invariant by $B_Q(\CV,\CF)$ with weight~$\boldlambda$.
The cone~$C_Q(\CV)$ is, by definition, the cone generated by the dominant weights~$\boldlambda$ such that $m(\boldlambda)>0$.
As discussed in \cref{subsec:intro rep theo}, results of Guillemin-Sternberg~\cites{GS1982qr,GS1982convex} and Mumford~\cite{NessMumford84}*{Appendix} identify the cone~$C_Q(\CV)$ with the image of a moment map modulo the coadjoint action.

The following result can be proved in more general situations by using Ressayre's dominant pairs~\cite{ressayre2010geometric} (see also~\cite{VW}).
We give a short proof in our setting.

\begin{proposition}\label{prp:cone}
Let $\CJ=(J_x)_{x\in Q_0}$, where $J_x=\{1,\dots,n_x\}$, and $\CV=\CV(\CJ)$.
Let $\boldlambda$ such that $V_{\boldlambda}$ occurs with nonzero multiplicity in~$\Sym^*(\CH_Q(\CV))$.
Then,
\begin{align}\label{eq:lambda eq}
  \sum_{x\in Q_0} \sum_{j=1}^{n_x} \lambda_x(j) = 0,
\end{align}
and for every $Q$-intersecting family of subsets~$\CK\subseteq_Q\CJ$ we have that
\begin{align}\label{eq:lambda ieq}
  \sum_{x\in Q_0} \sum_{k \in K_x} \lambda_x(k) \leq 0.
\end{align}
\end{proposition}
\begin{proof}
The first claim follows immediately from the fact that the element~$z\in\gl_Q(\CV)$ acts trivially on~$\Sym^*(\CH_Q(\CV))$.

For the second claim, let $\CK$ be a $Q$-intersecting family of subsets as above and let~$P$ be an arbitrary nonzero homogeneous polynomial that is semi-invariant by $B_Q(\CV,\CF)$ with weight~$\boldlambda$.
Let $\CS=\CV(\CK)$ and $\CT=\CV(\CK^c)$, where each~$(\CK^c)_x = K_x^c$, the complement of~$K_x$ in~$J_x=\{1,\dots,n_x\}$.
Consider the vector space from \cref{eq:fiber X(S)}:
\begin{align*}
  \mathbb X(\CS) = \{ v \in \CH_Q(\CV) : v \CS \subseteq \CS \}
\end{align*}
Since $\CS$ is $Q$-in\-ter\-sec\-ting, the $B_Q(\CV,\CF)$-orbit of $\mathbb X(\CS)$ is dense in $\CH_Q(\CV)$ (\cref{lem:generic}).
Thus, since~$P$ is nonzero and semi-invariant by~$B_Q(\CV,\CF)$, there must exist~$v\in\mathbb X(\CS)$ such that $P(v)\neq0$.
As an element of $\mathbb X(\CS)$, it is necessarily of the form
\begin{align*}
  v = \begin{pmatrix}v_{00} & v_{01} \\ 0 & v_{11}\end{pmatrix},
\end{align*}
where $v_{00} \in \CH_Q(\CS)$, $v_{01} \in \CH_Q(\CT, \CS)$, and $v_{11} \in \CH_Q(\CT)$.
Now consider the element $H=(H_x)_{x\in Q_0}$ in the Cartan subalgebra of~$\gl_Q(\CV)$ defined by~$H_x = \sum_{j\in K_x} h_{x,j}$ for $x\in Q_0$.
That is, each $H_x$ is of the form
\begin{align*}
  H_x = \begin{pmatrix} I & 0 \\ 0 & 0 \end{pmatrix}
\end{align*}
with respect to the direct sum~$V_x = S_x \oplus T_x$.
The orbit of~$v$ by the natural action of the one-parameter subgroup~$\exp(-tH)$ of~$\GL_Q(\CV)$ is given by~$v_t = (\exp(-t H_y) v_a \exp(t H_x))_{a\colon x\to y\in Q_1}$.
Thus,
\begin{align*}
  \lim_{t\to\infty} v_t
= \begin{pmatrix}v_{00} & 0 \\ 0 & v_{11}\end{pmatrix}.
\end{align*}
On the other hand, $P$ has weight~$\boldlambda$, so
\begin{align*}
  P(v_t) = e^{\braket{\boldlambda,H} t} P(v).
\end{align*}
We conclude that $\braket{\boldlambda,H}\leq0$, for otherwise the limit would not exist.
This inequality is exactly \cref{eq:lambda ieq}.
\end{proof}

Conversely, geometric invariant theory~\cite{ressayre2011geometric} implies that if $\boldlambda$ satisfies the conditions in \cref{eq:lambda ieq,eq:lambda eq} then it is an element of $C_Q(\CV)$ (see also~\cite{VW}).
Equivalently, in this case there exists a positive integer~$N\geq1$ such that $m(N\boldlambda)>0$.

In fact, we can choose~$N=1$, meaning that the semigroup of highest weights is saturated. 
For the Horn quiver, this was proved first by Knutson-Tao~\cite{MR1671451} and then by Derksen-Weyman~\cite{MR1758750}.
A geometric proof was given by Belkale~\cite{MR2177198} (see also~\cite{BVW}).
We thank Ressayre for explaining to us that, for a general quiver, this also follows from the Derksen-Weyman saturation theorem~\cite{MR1758750}, which asserts that, for a quiver~$Q$ without cycles, the semigroup of weights of semi-invariants is saturated (i.e., whenever there exists a semi-invariant of weight~$N\boldomega$ for some weight~$\boldomega$ and integer $N\geq1$, then there also exists a semi-invariant of weight~$\boldomega$).

Indeed, augment the quiver~$Q$ to a quiver~$\tilde Q$ and consider the family~$\tilde \CV=(\C^{n_{x,i}})$ of vector spaces with dimension vector~$\boldtilden$, as in \cref{subsec:augmented}.
To every family~$\boldtildeomega=(\tilde\omega_{x,i})$ of integers, we can associate a weight~$\boldlambda(\boldtildeomega)=(\lambda_x)$ for~$\GL_Q(\CV)$ by~$\lambda_x(i)=\sum_{j=i}^{n_x}\tilde\omega_{(x,j)}$.
Using the Cauchy formula for the decomposition of $\bigotimes_{i=1}^{n_x-1} \Sym^*(\Hom(\C^i,\C^{i+1}))$ under the action of $\prod_{i=1}^{n_x} GL(i)$, it is easy to see that if there exists a semi-invariant of weight~$\boldtildeomega$ for~$\CH_{\tilde Q}(\tilde\CV)$, then necessarily $\tilde\omega_{x,i}\geq 0$ for~$i=1,\dots,n_x-1$ and every $x\in Q_0$.
Thus, the corresponding~$\boldlambda(\boldtildeomega)$ is a dominant weight.
Conversely, any dominant weight~$\boldlambda$ can be written in this form for some~$\boldtildeomega$.
Furthermore, $\boldlambda(\boldomega)$ is in~$C_Q(\CV)$ if and only if $\boldtildeomega\in\Sigma_{\tilde Q}(\tilde\CV)$.
Consequently, the semigroup of highest weights for $\CH_Q(\CV)$ is saturated, since the semigroup of weights of semi-invariants for~$\CH_{\tilde Q}(\tilde\CV)$ is saturated.
The proof sketched above is similar to the Derksen-Weyman proof of the Horn inequalities~\cite{MR1758750}, which has been further simplified in~\cite{Cra-Bo-Ge}.

Let us discuss which among the inequalities in \cref{eq:lambda ieq} are irredundant.
In a general setting, geometric conditions for irredundancy were given by Ressayre in~\cite{ressayre2010geometric} and, in more detail for the particular case of quivers, in~\cite{MR2875833}.
For~$\CK$ to define an irredundant inequality, it must satisfy two conditions:
\label{page:ressayre conditions}
\begin{enumerate}[label=(R\arabic*)]
\item\label{it:ressayre cond 1} $\CV(\CK)$ belongs to $P_Q(\CV,\CF)$, i.e., the intersection variety~$\boldOmega(\CK)_v$ is generically reduced to a point,
\item\label{it:ressayre cond 2} $\dim C_Q(\CV(\CK))+\dim C_Q(\CV(\CK^c)) = \dim C_Q(\CV) - 1$, where~$\CK^c$ denotes the family of complements~$K_x^c$ of~$K_x$ in~$J_x=\{1,\dots,n_x\}$.
\end{enumerate}
For the Horn quiver~$H_s$, condition~\ref{it:ressayre cond 2} is a consequence of~\ref{it:ressayre cond 1}, but not in general (see end of \cref{sec:sun quiver}).

In practice, it can be difficult to determine when conditions~\ref{it:ressayre cond 1} and~\ref{it:ressayre cond 2} are satisfied.
It is often easier to use accelerated Fourier-Motzkin elimination on the complete (but, in general, redundant) set of inequalities associated to $Q$-intersecting~$\boldOmega(\CK)$ with $\edim_Q(\CK,\CJ)=0$ to obtain a complete set of irredundant inequalities characterizing the cone~$C_Q(\CV)$ (see also~\cite{VW}).

The cone $\Sigma_Q(\CV)$ is, by definition, the intersection of~$C_Q(\CV)$ with~$\z^*$.
Here, we embed $\z^*$ into the dual of the Lie algebra of the maximal torus of~$\GL_Q(\CV)$ via $\boldomega \mapsto \boldlambda$, where $\lambda_x(1)=\dots=\lambda_x(n_x)=\omega_x$.
We note that, for a general quiver~$Q$, this intersection can be reduced to~$\{0\}$.
We can characterize~$\Sigma_Q(\CV)$ by restricting a complete set of defining inequalities of the cone~$C_Q(\CV)$ to~$\z^*$, such as our \cref{eq:lambda eq,eq:lambda ieq}.
If $\CK=(K_x)$ is a family of subsets~$K_x\subseteq\{1,\dots,n_x\}$ and $\boldlambda$ as above, then $\sum_{k\in K_x}\lambda_x(k) = |K_x| \omega_x$.
Moreover, if $\CK$ is $Q$-intersecting, then $\alpha_x=|K_x|$ defines a Schofield subdimension vector~$\boldalpha$, and any Schofield subdimension vector of~$\boldn$ can be obtained in this way.
It follows that the cone~$\Sigma_Q(\CV)$ is determined by the inequalities
\begin{align*}
  \sum_{x\in Q_0} \alpha_x \omega_x \leq 0,
\end{align*}
where $\boldalpha$ ranges over all Schofield subdimension vectors of~$\boldn$, together with the equation
\begin{align*}
  \sum_{x\in Q_0} n_x \omega_x = 0.
\end{align*}
In this way, we recover the description of~$\Sigma_Q(\CV)$ due to Derksen-Weyman~\cite{MR1758750} and Schofield-van den Bergh~\cite{MR1908144}.
Irredundant inequalities are described in~\cite{MR1758750} when $\boldn$ is a Schur root.

\section{Sun Quiver}\label{sec:sun quiver}
We now discuss the `sun quiver' introduced in~\cite{Collins}:
\begin{center}
\begin{tikzpicture}\small
  \node (x1) at (0,1) {$1$};
  \node (x2) at (0.86602540378,0.5) {$2$};
  \node (x3) at (0.86602540378,-0.5) {$3$};
  \node (x4) at (0,-1) {$4$};
  \node (x5) at (-0.86602540378,-0.5) {$5$};
  \node (x6) at (-0.86602540378,0.5) {$6$};
  \draw[->] (x1) edge (x2);
  \draw[->] (x3) edge (x2);
  \draw[->] (x3) edge (x4);
  \draw[->] (x5) edge (x4);
  \draw[->] (x5) edge (x6);
  \draw[->] (x1) edge (x6);
\end{tikzpicture}
\end{center}
The sun quiver has a discrete rotation symmetry ($x \mapsto x+2$) and a reflection symmetry that interchanges $2\leftrightarrow6$ and $3\leftrightarrow5$.

The family $\CJ=(\{1,2\},\dots,\{1,2\})$ and its dimension vector~$(2,\dots,2)$ respect both symmetries.
We use \cref{thm:horn concrete} to compute the $Q$-intersecting subfamilies~$\CK\subseteq_Q\CJ$.
Up to symmetry, there are 113 subfamilies, corresponding to 39 Schofield subdimension vectors. 
The latter are given by the following list:
\begin{align*}
&(0, 0, 0, 0, 0, 0),&
&(0, 0, 0, 0, 0, 1),&
&(0, 0, 0, 0, 0, 2),&
&(0, 0, 0, 1, 0, 1),&\\
&(0, 0, 0, 1, 0, 2),&
&(0, 0, 0, 1, 1, 1),&
&(0, 0, 0, 1, 1, 2),&
&(0, 0, 0, 2, 0, 2),&\\
&(0, 0, 0, 2, 1, 2),&
&(0, 0, 0, 2, 2, 2),&
&(0, 1, 0, 1, 0, 1),&
&(0, 1, 0, 1, 0, 2),&\\
&(0, 1, 0, 1, 1, 1),&
&(0, 1, 0, 1, 1, 2),&
&(0, 1, 0, 2, 0, 2),&
&(0, 1, 0, 2, 1, 2),&\\
&(0, 1, 0, 2, 2, 2),&
&(0, 1, 1, 1, 0, 2),&
&(0, 1, 1, 1, 1, 1),&
&(0, 1, 1, 1, 1, 2),&\\
&(0, 1, 1, 2, 0, 2),&
&(0, 1, 1, 2, 1, 1),&
&(0, 1, 1, 2, 1, 2),&
&(0, 1, 1, 2, 2, 2),&\\
&(0, 2, 0, 2, 0, 2),&
&(0, 2, 0, 2, 1, 2),&
&(0, 2, 0, 2, 2, 2),&
&(0, 2, 1, 1, 1, 2),&\\
&(0, 2, 1, 2, 1, 2),&
&(0, 2, 1, 2, 2, 2),&
&(0, 2, 2, 2, 2, 2),&
&(1, 1, 1, 1, 1, 1),&\\
&(1, 1, 1, 1, 1, 2),&
&(1, 1, 1, 2, 1, 2),&
&(1, 1, 1, 2, 2, 2),&
&(1, 2, 1, 2, 1, 2),&\\
&(1, 2, 1, 2, 2, 2),&
&(1, 2, 2, 2, 2, 2),&
&(2, 2, 2, 2, 2, 2).&
\end{align*}
Up to symmetry, there are 59 $Q$-intersecting subfamilies~$\CK$ that satisfy the condition~$\edim_Q(\CK,\CJ)=0$.
They are given by
\begin{align*}
&(\emptyset,\emptyset,\emptyset,\emptyset,\emptyset,\emptyset),&
&(\emptyset,\emptyset,\emptyset,\emptyset,\emptyset,1),&
&(\emptyset,\emptyset,\emptyset,\emptyset,\emptyset,12),&\\
&(\emptyset,\emptyset,\emptyset,1,\emptyset,1),&
&(\emptyset,\emptyset,\emptyset,1,\emptyset,12),&
&(\emptyset,\emptyset,\emptyset,1,2,2),&\\
&(\emptyset,\emptyset,\emptyset,2,1,2),&
&(\emptyset,\emptyset,\emptyset,1,2,12),&
&(\emptyset,\emptyset,\emptyset,2,1,12),&\\
&(\emptyset,\emptyset,\emptyset,12,\emptyset,12),&
&(\emptyset,\emptyset,\emptyset,12,1,12),&
&(\emptyset,\emptyset,\emptyset,12,12,12),&\\
&(\emptyset,1,\emptyset,1,\emptyset,1),&
&(\emptyset,1,\emptyset,1,\emptyset,12),&
&(\emptyset,1,\emptyset,1,2,2),&\\
&(\emptyset,1,\emptyset,2,1,2),&
&(\emptyset,1,\emptyset,1,2,12),&
&(\emptyset,1,\emptyset,2,1,12),&\\
&(\emptyset,1,\emptyset,12,\emptyset,12),&
&(\emptyset,1,\emptyset,12,1,12),&
&(\emptyset,1,\emptyset,12,12,12),&\\
&(\emptyset,1,2,2,\emptyset,12),&
&(\emptyset,2,1,2,\emptyset,12),&
&(\emptyset,1,2,2,2,2),&\\
&(\emptyset,2,1,2,2,2),&
&(\emptyset,2,2,1,2,2),&
&(\emptyset,1,2,2,2,12),&\\
&(\emptyset,2,1,2,2,12),&
&(\emptyset,2,2,1,2,12),&
&(\emptyset,2,2,2,1,12),&\\
&(\emptyset,1,2,12,\emptyset,12),&
&(\emptyset,2,1,12,\emptyset,12),&
&(\emptyset,1,2,12,1,2),&\\
&(\emptyset,1,2,12,2,1),&
&(\emptyset,2,1,12,1,2),&
&(\emptyset,1,2,12,1,12),&\\
&(\emptyset,2,1,12,1,12),&
&(\emptyset,1,2,12,12,12),&
&(\emptyset,2,1,12,12,12),&\\
&(\emptyset,12,\emptyset,12,\emptyset,12),&
&(\emptyset,12,\emptyset,12,1,12),&
&(\emptyset,12,\emptyset,12,12,12),&\\
&(\emptyset,12,1,2,2,12),&
&(\emptyset,12,2,1,2,12),&
&(\emptyset,12,1,12,1,12),&\\
&(\emptyset,12,1,12,12,12),&
&(\emptyset,12,12,12,12,12),&
&(2,2,2,2,2,2),&\\
&(1,2,2,2,2,12),&
&(2,1,2,2,2,12),&
&(2,2,1,2,2,12),&\\
&(1,2,2,12,1,12),&
&(2,1,2,12,1,12),&
&(1,2,2,12,12,12),&\\
&(2,1,2,12,12,12),&
&(1,12,1,12,1,12),&
&(1,12,1,12,12,12),&\\
&(1,12,12,12,12,12),&
&(12,12,12,12,12,12),&
\end{align*}
where we again write $12$ instead of $\{1,2\}$ etc.~to improve readability.
For example, $(1,2,2,12,1,12)$ and $(2,1,2,12,1,12)$ are two (inequivalent) subfamilies that both correspond to the Schofield subdimension vector~$(1, 1, 1, 2, 1, 2)$.

We now compute the polyhedral cone characterizing the highest weights~$\boldlambda$ that appear in~$\Sym^*(\CH_Q(\CV))$, where $\CV=(\C^2,\dots,\C^2)$.
It is defined by the constraints in \cref{prp:cone} and the Weyl chamber inequalities~$\lambda_x(1)\geq\lambda_x(2)$ for each vertex~$x$.
The resulting cone has 36 extreme rays and 75 faces.
In addition to the Weyl chamber inequalities and the constraint $\sum_{x=1}^6 \sum_{a=1}^2 \lambda_x(a)=0$, a minimal complete set of defining inequalities is (up to symmetry) given by the following list
\begin{align*}
&\lambda_1(1) + \lambda_2(2) + \lambda_6(2) \leq 0,\\
&\lambda_1(2) + \lambda_2(1) + \lambda_6(2) \leq 0,\\
&\lambda_1(1) + \lambda_2(2) + \lambda_3(2) + \lambda_4(2) + \lambda_6(2) \leq 0,\\
&\lambda_1(2) + \lambda_2(1) + \lambda_3(2) + \lambda_4(2) + \lambda_6(2) \leq 0,\\
&\lambda_1(2) + \lambda_2(1) + \lambda_4(2) + \lambda_5(2) + \lambda_6(2) \leq 0,\\
& |\lambda_1| +  |\lambda_2| +  |\lambda_6| \leq 0,\\
&\lambda_1(1) +  |\lambda_2| + \lambda_3(1) + \lambda_4(2) + \lambda_6(2) \leq 0,\\
&\lambda_1(1) +  |\lambda_2| + \lambda_3(2) + \lambda_4(1) + \lambda_6(2) \leq 0,\\
&\lambda_1(2) +  |\lambda_2| + \lambda_3(2) + \lambda_4(1) + \lambda_6(1) \leq 0,\\
&\lambda_1(1) +  |\lambda_2| + \lambda_3(2) + \lambda_4(2) + \lambda_5(2) + \lambda_6(2) \leq 0,\\
&\lambda_1(1) + \lambda_2(2) + \lambda_3(2) +  |\lambda_4| + \lambda_5(2) + \lambda_6(2) \leq 0,\\
&\lambda_1(2) +  |\lambda_2| + \lambda_3(2) + \lambda_4(1) + \lambda_5(2) + \lambda_6(2) \leq 0,\\
& |\lambda_1| +  |\lambda_2| + \lambda_3(1) + \lambda_4(2) + \lambda_5(2) +  |\lambda_6| \leq 0,\\
& |\lambda_1| +  |\lambda_2| + \lambda_3(2) + \lambda_4(1) + \lambda_5(2) +  |\lambda_6| \leq 0,
\end{align*}
together with $\lambda_x(2)\geq0$ for odd~$x$ and $\lambda_x(1)\leq0$ for even~$x$.
We computed these inequalities using Fourier-Motzkin elimination starting from the conditions in \cref{prp:cone} for $Q$-intersecting families~$\CK$ with expected dimension zero and the Weyl chamber inequalities.
The above list coincides with Collins' updated result~\cite{Collins}, obtained by using the isomorphism between~$C_Q(\CV)$ and~$\Sigma_{\tilde Q}(\tilde\CV)$
described in \cref{sec:rep theo} and the Derksen-Weyman description of irredundant inequalities for~$\Sigma_{\tilde Q}(\tilde \CV)$ in terms of decompositions into Schur roots.

In this simple case, it is also feasible to apply (and verify) Ressayre's criterion for irredundancy.
All families~$\CK$ listed above satisfy Ressayre's condition~\ref{it:ressayre cond 1} on p.~\pageref{page:ressayre conditions}, except for the family~$\CK=(\{2\},\{2\},\{2\},\{2\},\{2\},\{2\})$, which leads to a variety~$\boldOmega(\CK)_v$ which generically consists of two points.
(Generically, the composition $v_{1\to6}^{-1} v_{5\to 6} v_{5\to4}^{-1} v_{3\to4} v_{3\to2}^{-1} v_{1\to2}$ has two one-dimensional eigenspaces~$S_1$, each of which gives rise to a point~$\CS\in\boldOmega(\CK)_v$.)
The corresponding inequality~$\sum_x \lambda_x(2) \leq 0$ indeed follows by adding the Weyl chamber inequalities~$\lambda_x(2)-\lambda_x(1)\leq 0$ to the equation~$\sum_x \lambda_x(1)+\lambda_x(2)=0$.

It is also not hard to see that if~$\CK$ is a family for which the undirected subgraph of the sun quiver obtained by erasing the vertices corresponding to empty sets (i.e., $K_x=\emptyset$) is a disconnected graph, then~$\CK$ (and also~$\CK^c$) do not satisfy Ressayre's condition~\ref{it:ressayre cond 2} for irredundancy.
For example, the inequalities $\lambda_4(1)\leq0$ and $\lambda_6(1)\leq0$ are irredundant inequalities associated to $(\emptyset,\emptyset,\emptyset,\{1\},\emptyset, \emptyset)$ and $(\emptyset,\emptyset,\emptyset,\emptyset,\emptyset,\{1\})$), respectively.
In contrast, the family $\CK=(\emptyset,\emptyset,\emptyset,\{1\},\emptyset,\{1\})$ satisfies condition~\ref{it:ressayre cond 1} but not condition~\ref{it:ressayre cond 2}, and the corresponding inequality~$\lambda_4(1)+\lambda_6(1)\leq0$ is redundant.

\emph{A priori} conditions for irredundancy have been given by Belkale-Kumar~\cite{BelkaleKumar}, Derksen-Weyman~\cite{MR1758750}, Knutson-Tao-Woodward~\cite{KnutsonTaoWoodward}, and Ressayre~\cite{ressayre2010geometric} in terms of Schubert calculus (for $GL(n)$, equivalently, in terms of Littlewood-Richardson coefficients).
They appear to be difficult to test in practice.

\section*{Acknowledgements}
It is a pleasure to acknowledge discussions with Giovanni Cerulli Irelli, Evgeny Feigin, Bernhard Keller, Shrawan Kumar, and Nicolas Ressayre.

V.~Baldoni acknowledges the MIUR Excellence Department Project awarded to the Department of Mathematics, University of Rome Tor Vergata, CUP E83C18000100006 and the partial support of a PRIN2015 grant.
M.~Walter acknowledges support by the NWO through Veni grant~680-47-459 and grant OCENW.KLEIN.267, by the Deutsche Forschungsgemeinschaft (DFG, German Research Foundation) under Germany's Excellence Strategy - EXC\ 2092\ CASA - 390781972, and by the European Research Council~(ERC) through ERC Starting Grant 101040907-SYMOPTIC.



\begin{bibdiv}
\begin{biblist}
\bib{quivershort}{article}{
  author={Baldoni, Velleda},
  author={Vergne, Mich\`{e}le},
  author={Walter, Michael},
  title={Horn inequalities and quivers},
  journal={\href{https://arxiv.org/abs/1804.00431}{arXiv:1804.00431}},
}

\bib{MR2177198}{article}{
  author={Belkale, Prakash},
  title={Geometric proofs of Horn and saturation conjectures},
  journal={J. Algebraic Geom.},
  volume={15},
  date={2006},
  pages={133--173},
}

\bib{BelkaleKumar}{article}{
  author={Belkale, Prakash},
  author={Kumar, Shrawan},
  title={Eigenvalue problem and a new product in cohomology of flag varieties},
  journal={Invent. Math.},
  volume={166},
  pages={185--228},
  year={2006},
}

\bib{BVW}{article}{
  author={Berline, Nicole},
  author={Vergne, Mich\`{e}le},
  author={Walter, Michael},
  title={The Horn inequalities from a geometric point of view},
  journal={Enseign. Math.},
  volume={63},
  pages={403--470},
  year={2017},
}
\bib{BR}{article}{
  author={Bertozzi, Maria},
  author={Reineke, Mark},
  title={Momentum map images of representation spaces of quivers},
  journal={\href{https://arxiv.org/abs/2001.08071v2}{arXiv:2001.08071v2}},
}


\bib{QuiverCellular}{article}{
  author={Giovanni Cerulli Irelli},
  author={Evgeny Feigin},
  author={Markus Reineke},
  title={Quiver Grassmannians and degenerate flag varieties},
  journal={Algebra Number Theory},
  volume={6},
  year={2012},
  pages={165--194},
}

\bib{Collins}{article}{
  author={Collins, Brett},
  title={Generalized Littlewood-Richardson coefficients for branching rules of $\GL(n)$ and extremal weight crystals},
  journal={\href{https://arxiv.org/abs/1801.09170v3}{arXiv:1801.09170v3}},
}

\bib{MR1834739}{article}{
   author={Crawley-Boevey, William},
   title={Geometry of the moment map for representations of quivers},
   journal={Compositio Math.},
   volume={126},
   date={2001},
   number={3},
   pages={257--293},
}


\bib{Cra-Bo-Ge}{inproceedings}{
  author={Crawley-Boevey,  William},
  author={Geiss, Christof},
  title={Horn's problem and semi-stability for quiver representations},
  booktitle={Proceedings of the Ninth International Conference on Representations of Algebras, Beijing Normal University},
  publisher={Beijing Normal University Press},
  pages={40--48},
  year={2002},
}


\bib{MR1758750}{article}{
  author={Derksen, Harm},
  author={Weyman, Jerzy},
  title={Semi-invariants of quivers and saturation for Littlewood-Richardson coefficients},
  journal={J. Amer. Math. Soc.},
  volume={13},
  date={2000},
  pages={467--479},
}

\bib{DSW}{article}{
  author={Derksen, Harm},
  author={Schofield, Aidan},
  author={Weyman, Jerzy},
  title={On the number of subrepresentations of a general quiver representation},
  journal={J. London Math. Soc.},
  volume={76},
  pages={135--147},
  year={2007},
}

\bib{GS1982convex}{article}{
  author={Guillemin, Victor},
  author={Sternberg, Shlomo},
  title={Convexity properties of the moment mapping},
  journal={Invent. math.},
  volume={67},
  date={1982},
  pages={419--513},
}

\bib{GS1982qr}{article}{
  author={Guillemin, Victor},
  author={Sternberg, Shlomo},
  title={Geometric quantization and multiplicities of group representations},
  journal={Invent. math.},
  volume={67},
  date={1982},
  pages={515--538},
}

\bib{MR0140521}{article}{
  author={Horn, Alfred},
  title={Eigenvalues of sums of Hermitian matrices},
  journal={Pacific J. Math.},
  volume={12},
  date={1962},
  pages={225--241},
}

\bib{MR1315461}{article}{
   author={King, A. D.},
   title={Moduli of representations of finite-dimensional algebras},
   journal={Quart. J. Math. Oxford Ser. (2)},
   volume={45},
   date={1994},
   number={180},
   pages={515--530},
}


\bib{MR1671451}{article}{
  author={Knutson, Allen},
  author={Tao, Terence},
  title={The honeycomb model of $\GL_n(\C)$ tensor products. I. Proof of the saturation conjecture},
  journal={J. Amer. Math. Soc.},
  volume={12},
  date={1999},
  pages={1055--1090},
}

\bib{KnutsonTaoWoodward}{article}{
  author={Knutson, Allen},
  author={Tao, Terence},
  author={Woodward, Chris},
  title={The honeycomb model of $\GL_n(\C)$ tensor products. II. Puzzles determine facets of the Littlewood-Richardson cone},
  journal={J. Amer. Math. Soc.},
  volume={17},
  date={2004},
  pages={19--48},
}

\bib{NessMumford84}{article}{
  author={Ness, Linda},
  title={A stratification of the null cone via the moment map},
  note={With an appendix by David Mumford},
  journal={Amer. J. Math.},
  volume={106},
  date={1984},
  pages={1281--1329},
}

\bib{reineke2013every}{article}{
  title={Every projective variety is a quiver Grassmannian},
  author={Reineke, Markus},
  journal={Algebr. Represent. Theory},
  pages={1--2},
  year={2013},
}

\bib{MR2875833}{article}{
  author={Ressayre, Nicolas},
  title={GIT-cones and quivers},
  journal={Math. Z.},
  volume={270},
  date={2012},
  pages={263--275},
}

\bib{ressayre2010geometric}{article}{
  title={Geometric invariant theory and the generalized eigenvalue problem},
  author={Ressayre, Nicolas},
  journal={Invent. Math.},
  volume={180},
  pages={389--441},
  year={2010},
  publisher={Springer},
}

\bib{ressayre2011geometric}{article}{
  title={Geometric invariant theory and generalized eigenvalue problem II},
  author={Ressayre, Nicolas},
  journal={Ann. Inst. Fourier},
  volume={61},
  pages={1467--1491},
  year={2011}
}

\bib{ressayrepc}{unpublished}{
  title={private communication},
  author={Ressayre, Nicolas},
  year={2018}
}

\bib{MR1162487}{article}{
  author={Schofield, Aidan},
  title={General representations of quivers},
  journal={Proc. London Math. Soc. (3)},
  volume={65},
  date={1992},
  pages={46--64},
}


\bib{MR1908144}{article}{
  author={Schofield, Aidan},
  author={van den Bergh, Michel},
  title={Semi-invariants of quivers for arbitrary dimension vectors},
  journal={Indag. Math. (N.S.)},
  volume={12},
  date={2001},
  pages={125--138},
}

\bib{sherman1}{article}{
  author={Sherman, Cass},
  title={Geometric Proof of a Conjecture of King, Tollu, and Toumazet},
  journal={\href{https://arxiv.org/abs/1505.06551}{arXiv:1505.06551}},
}

\bib{MR3633317}{article}{
  author={Sherman, Cass},
  title={Quiver generalization of a conjecture of King, Tollu, and Toumazet},
  journal={J. Algebra},
  volume={480},
  date={2017},
  pages={487--504},
}

\bib{VW}{article}{
  author={Vergne, Mich\`ele},
  author={Walter, Michael},
  title={Inequalities for Moment Cones of Finite-Dimensional Representations},
  journal={J. Symplectic Geom.},
  volume={15},
  date={2017},
  pages={1209--1250},
}
\end{biblist}
\end{bibdiv}
\end{document}